\documentclass[11pt]{article}

\usepackage[top=1in, bottom=1in, outer=1in, inner=1in, heightrounded, marginparwidth=4cm, marginparsep=1cm]{geometry}
\makeatother
\usepackage{xcolor}

\usepackage{pgfplots}
\usepackage{mathrsfs}
\usepackage[colorlinks=true, linkcolor=orange, citecolor=green!60!black, urlcolor=blue,draft=false]{hyperref}
\usepackage{amsfonts, amsmath,amssymb, amsthm, enumerate, fancyhdr , authblk,mathtools,mathrsfs, varwidth} 

\newcommand{\tmm}[1]{{\usefont{OT1}{phv}{b}{n}\tiny #1}}

\DeclareFontFamily{OT1}{pzc}{}
\DeclareFontShape{OT1}{pzc}{m}{it}{<-> s * [1.2] pzcmi7t}{}
\DeclareMathAlphabet{\mpzc}{OT1}{pzc}{m}{it}

\usepackage{sectsty} 
\allsectionsfont{\color{RoyalBlue}\usefont{OT1}{phv}{b}{n}} 

\numberwithin{equation}{section}
\newtheoremstyle{RoyalBlue}{}{}{\itshape\color{RoyalBlue}}{}{\color{RoyalBlue}\sffamily\bfseries}{.}{ }{}
\newtheoremstyle{RoyalBlue_def}{}{}{}{}{\color{RoyalBlue}\bfseries}{.}{ }{}
\theoremstyle{RoyalBlue}

\newtheorem{Thm}{Theorem}[section]
\newtheorem{Lem}{Lemma}[section]

\theoremstyle{RoyalBlue_def}

\newtheorem{Rmk}{Remark}[section]

\def\ep{\epsilon}
\def\f{\frac}

\def\grad{\nabla}

\newcommand{\abs}[1]{ \left| #1 \right|}
\newcommand{\norm}[1]{\lVert#1\rVert}

\newcommand{\inner}[2]{\langle #1, #2 \rangle}

\newcommand{\Rmnum}[1]{ \uppercase\expandafter{\romannumeral  #1}}
\newcommand{\mc}[1]{\mathcal{#1}}
\newcommand{\ubrace}[2]{\begingroup
      \color{RoyalBlue}
      \underbrace{\color{black}#1}_{#2}
      \endgroup
      }
\newcommand{\vl} {\boldsymbol{\ell}}
\newcommand{\vlambda} {\boldsymbol{\lambda}}
\newcommand{\vxi} {\boldsymbol{\xi}}

\newcommand{\vu} {\mathbf{u}}
\newcommand{\vv} {\mathbf{v}}

\newcommand{\vq} {\mathbf{q}}
\newcommand{\vp} {\mathbf{p}}
\newcommand{\vP} {\mathbf{P}}

\newcommand{\vo} {\mathbf{0}}

\newcommand{\vphi} {\boldsymbol \varphi}

\def\dps{\displaystyle}

\newcommand{\wbar}[1]{\overline{#1}}

\definecolor{RoyalBlue}{HTML}{1462AA}
\definecolor{Cyan}{HTML}{007AA2}
\definecolor{rred}{rgb}{0.7,0,0.1}
\definecolor{questionred}{rgb}{0.2,0.2,0.8}

 \usepackage{graphicx, enumerate,  bbm, graphicx, calligra, lastpage}

 \date{\empty}

 \usepackage{fancyhdr} 

 \usepackage{lettrine} 
 \newcommand{\initial}[2]{ 
 \lettrine[lines=2,lhang=0.3,nindent=0em]{
 \color{RoyalBlue}
 {\textsf{#1}}}{#2}}

\usepackage{footmisc}
\DefineFNsymbols{mySymbols}{{\ensuremath\dagger}{\ensuremath\ddagger}\S\P
   *{**}{\ensuremath{\dagger\dagger}}{\ensuremath{\ddagger\ddagger}}}
\setfnsymbol{mySymbols}

\fancypagestyle{titlepage}{

\cfoot{\color{RoyalBlue} \tiny \usefont{OT1}{phv}{b}{n} Page \thepage\ of \pageref{LastPage}}
   \rfoot{}
}

\begin{document}

\title{\sffamily\bfseries\color{RoyalBlue}Weak solutions of the Shigesaka-Kawasaki-Teramoto equations and their attractors}
\author[1]{{\color{RoyalBlue}\sffamily Du Pham}\thanks{\color{RoyalBlue}du.pham@utsa.edu}}
\author[2]{{\color{RoyalBlue}\sffamily Roger Temam}\thanks{\color{RoyalBlue}temam@indiana.edu}}
\affil[1]{\footnotesize {Department of Mathematics, University of Texas at San Antonio, One UTSA Circle, San Antonio, Texas 78249, U.S.A.} }
\affil[2]{\footnotesize {The Institute for Scientific Computing and Applied Mathematics, Indiana University, 831 East Third Street, Rawles Hall, Bloomington, Indiana 47405, U.S.A.}}

\renewcommand\Authand{ \& }
\renewcommand\Authands{ \& }

\renewcommand\Authands{ and }

\maketitle
\thispagestyle{titlepage}

\begin{abstract} \initial{\textcolor{RoyalBlue}{W}}{e} derive the global existence of weak solutions of the Shigesada-Kawasaki-Teramoto systems in space dimension $d\le 4$ with a rather general condition on the coefficients. The  existence is established using finite differences in time with truncations and an argument of Stampachia's maximum principle to show the positivity of the solutions. We derive also the existence of a weak global attractor. \end{abstract}
  \lhead{\color{RoyalBlue}\footnotesize{\tmm{Weak solutions of the SKT equations and their attractors}}}
   \rhead{\color{RoyalBlue}\footnotesize{\tmm{D. Pham \& R. Temam }}
   \cfoot{\color{RoyalBlue} \tiny \usefont{OT1}{phv}{b}{n} Page \thepage\ of \pageref{LastPage}}
   \rfoot{}
   }

  {\footnotesize

  }

\tableofcontents

\section{Introduction\label{sec:intro}}
\initial{\textcolor{RoyalBlue}{W}}{e} let $\Omega \subset \mathbb{R}^d$, $d\le 4$, be an open bounded set and denote by $\Omega_T=\Omega\times (0,T)$.  We look for solutions to the following Shigesada-Kawasaki-Teramoto (SKT) system of diffusion reaction equations, see \cite{SKT79} for the setup of the system:
\begin{align} \label{SKT system}
& \begin{cases}
  &\partial_t u - \Delta p_1(u,v) + q_1(u,v) =\ell_1(u)\text{ in } \Omega_T,\\
 &\partial_t v - \Delta p_2(u,v) + q_2(u,v) =\ell_2(u) \text{ in } \Omega_T, \\
   \end{cases}
\\& \label{bry cond and init cond}
\begin{cases}
 & \partial_n u = \partial_n v =0\text{ on } \partial \Omega \times (0,T) \text{ or }u=v=0 \text{ on }\partial \Omega\times (0,T), \\
 &u(x,0) = u_0(x) \ge 0, \; v(x,0) = v_0(x) \ge 0 \text{ in } \Omega,
\end{cases}
\end{align}

where

\begin{subequations}
 \label{pi qi li}
 \begin{align}
&p_1(u,v) = (d_1 + a_{11} u+ a_{12}v) u, \\
& p_2(u,v) = (d_2 + a_{21} u+ a_{22}v) v, \\
&q_1(u,v) = (b_1u+c_1v)u, \, \ell_1(u) = a_1u,\\
& q_2(u,v) = (b_2u+c_2v)v,\, \ell_2(v) = a_2v,
 \end{align}
\end{subequations}
with $a_{ij} \ge 0, b_i \ge 0, c_i \ge 0, a_i \ge 0, d_i \ge 0$.

We will consider also the alternate form of \eqref{SKT system}
\begin{equation} \label{SKT alternate}
 \begin{cases}
  &\dps \partial_t u - \grad \cdot \big[(d_1+2a_{11}u+a_{12}v)\grad u + a_{12}u    \grad v\big]  +q_1(u,v) = \ell_1(u)\text{ in } \Omega_T,\\
  &\dps \dps \partial_t v - \grad \cdot \big[  a_{21}v  \grad u+ (d_2+2a_{21}u+a_{22}v)\grad v\big]  +q_2(u,v) = \ell_2(v) \text{ in } \Omega_T.
    \end{cases}
\end{equation}

The  mathematical literature of the SKT system \eqref{SKT system} is vast; here are some highlights. In the case of weak cross-diffusion or triangular systems (when $a_{12}=0$ or $a_{21} = 0$), there have been a series of developments: when $d = 2$, Lou et al. \cite{LNW98}
proved the existence of smooth solutions to the SKT model. The method in \cite{LNW98} can also be modified
to cover the case $d = 1$. Choi et al. \cite{CLY04} and Le et al. \cite{LNN03} independently settled the case $d \le 5$. Tuoc
\cite{Tuo07} proved the existence of smooth solutions to the SKT model when $d \le 9$. Recently, Hoang et al.
\cite{HNP15} established the global existence of smooth solutions for any space dimension, which improves results in \cite{Ngu06}.

In the case of full systems (when $a_{12},a_{21}\not =0$), for  solutions achieved by a maximum principle, interpolation and Sobolev embeddings, Yagi in \cite{Yag93} proved the global existence of solutions in the two  dimensional case  under the condition $8a_{11} > a_{21}, \; 8a_{22} > a_{12}$ by using a maximum principle developed for strong solutions $u,v \in \mathcal{C}([0,T];H^{1+\ep}(\Omega)) \cap \mathcal{C}((0,T];H^2(\Omega)) \cap \mathcal{C}^1((0,T];L^2(\Omega))$.

In \cite{Le13,Le15}, the author uses a different approach by controlling the BMO
norms instead of using the classical results of Amann \cite{Ama90,Ama89}
that require estimates for the H\"{o}lder norms of the solutions.

In this article we aim to establish the existence of weak solutions to the
SKT system \eqref{SKT system} in space dimension $d\le 4$, when the
coefficients satisfy
\begin{subequations}
 \label{coef cond}
 \begin{align} \label{coef cond a}
&4a_{11} > a_{12}, \quad 4a_{22}> a_{21},\\
&2a_{21}> a_{12}> \f12 a_{21}.\label{coef cond b}
 \end{align}
\end{subequations}
Our work relates to that of Yagi, \cite{Yag93}, however we deal
with dimensions $d\le 4$ while \cite{Yag93} is limited to dim $d=1,2$,
and our hypotheses \eqref{coef cond} are more general, perhaps the
most general possible in this context.

Another difference is that our proof is completely self contained, while
\cite{Yag93} relies on some earlier articles \cite{Yag88,Yag90,Yag91}.
Our proof relies on seemingly new a priori estimates. Finally another
novelty is to contract a weak global attractor for these equations in the
spirit of Ball \cite{Ball97}, Sell
\cite{Sell96} , \cite{FT87}, \cite{FMRT01}, and  \cite{FRT10}. As in these references the concept of attractor is weak,
because we do not establish the uniqueness of weak solutions, like for
the 3D Navier-Stokes equations. Additional regularity of solutions and
possibly the issue of uniqueness will be studied in a subsequent work.

{\bf \it The authors of this article are not specialists of reaction-diffusion equations and they will welcome any bibliographical reference from the Editors or Referees that they may have overlooked.}

%
%

We will work with the following vector form of \eqref{SKT alternate}:
\begin{equation}
  \label{SKT vector}
   \partial_t \vu - \grad \cdot \Big( \vP(\vu) \grad \vu \Big) + \vq (\vu) =\vl (\vu),
\end{equation}
where $\vu =(u,v), \vq(\vu) = (q_1(\vu),q_2(\vu)), \vl(\vu) = (\ell_1(\vu),\ell_2(\vu))$ and
\begin{equation}
 \label{P}
 \vP(\vu) = \begin{pmatrix}
             p_{11}(u,v) & p_{12}(u,v)\\
             p_{21}(u,v) & p_{22}(u,v)
            \end{pmatrix}
            =\begin{pmatrix} d_1 +2a_{11}u + a_{12}v & a_{12} u \\ a_{21}v & d_2 + a_{21}u +2a_{22}v \end{pmatrix}.
\end{equation}

We consider later on the mapping
\begin{equation}
  \label{P map} \mc{P}: \vu=(u,v) \mapsto \vp=(p_1,p_2),
\end{equation}
and we observe that
\begin{equation}
  \label{P jac} \vP(\vu) = \f{D\mc{P}}{D \vu}(\vu).
\end{equation}
\section{Formal a priori estimates\label{sec: apriori est}}
We derive formal a priori estimates for the solutions of \eqref{SKT system}-\eqref{SKT vector}, assuming that $u,v\ge 0$
are sufficiently smooth.

\subsection{First a priori estimates and conditions on the $a_{ij}$\label{sec: first apriori est}}
By multiplying \eqref{SKT vector} by $\vu$, integrating over $\Omega$, and integrating by parts using the boundary condition
\eqref{bry cond and init cond}$_1$, we find
\begin{equation} \label{energy weak}
 \f 12 \f{d}{dt} \abs{\vu}^2 + \inner{\vP(\vu)\grad \vu}{\grad \vu} + \inner{\vq(\vu)}{\vu} = \inner{\vl(\vu)}{\vu}.
\end{equation}

We concentrate first on the term $ \inner{\vP(\vu)\grad \vu}{\grad \vu} $ and for $\vxi =(\xi_1,\xi_2) \in \mathbb{R}^2$ we write
\begin{multline} \label{2.10}
\big(\vP(\vu) \vxi\big) \cdot \vxi=
\big( d_1 + 2a_{11}u +a_{12}v \big) \xi_1^2 + a_{12}u \xi_1 \xi_2  + a_{21}v \xi_1\xi_2 +\big( d_2 + a_{21}u +a_{22}v)\xi_2^2
\\ =  d_1\xi_1^2 + d_2\xi_2^2 + \big(2a_{11} u + a_{12}v \big) \xi_1^2  + \big(a_{12}u +  a_{21}v\big) \xi_1\xi_2
 + \big(a_{21} u + 2 a_{22}v \big)\xi_2^2.
\end{multline}

When the conditions \eqref{coef cond a}, \eqref{coef cond b} are satisfied, we can prove that the matrix $\vP(\vu)$ is (pointwise)  positive definite and that:
\begin{equation} \label{P positive definite}
\left(\vP(\vu)\vxi \right) \cdot \vxi \ge \alpha(u+v)\abs{\vxi}^2 + d_0 \abs{\vxi}^2,
\end{equation}
where $d_0 = \min (d_1,d_2)$ and  \begin{equation}\label{alpha}
\alpha= \min \left\{2a_{11}-\f12a_{12},2a_{22}-\f12 a_{21}, a_{12}-\f12 a_{21},a_{21}-\f 12 a_{12}\right\} >0.
\end{equation}

Indeed, we bound from below the  term  $\xi_1\xi_2$ in \eqref{2.10} as follows:
\[
a_{12} u\xi_1 \cdot \xi_2 + a_{21}v\xi_1 \cdot \xi_2
 \ge -\f 12 a_{12}u \abs{\xi_1}^2 -\f 12 a_{12}u \abs{\xi_2}^2
-\f 12 a_{21}v \abs{\xi_1}^2 -\f 12 a_{21}v \abs{\xi_2}^2.
\]
Hence, the sum of the diffusion terms is bounded as follows
\begin{multline*}
\big(\vP(\vu) \vxi\big) \cdot \vxi \ge  d_1\abs{\xi_1}^2 + d_2\abs{\xi_2}^2 + (2a_{11}-\f12a_{12})u \abs{\xi_1}^2  + (a_{12}-\f12 a_{21})v \abs{\xi_1}^2
\\ + (2a_{22}-\f12 a_{21})v \abs{\xi_2}^2 + (a_{21}-\f 12 a_{12})u \abs{\xi_2}^2
\\ \ge d_1\abs{\xi_1}^2 + d_2\abs{\xi_2}^2  + \alpha  (u + v) (\abs{\xi_1}^2 + \abs{\xi_2}^2).
\end{multline*}

\begin{Rmk}

 When $a_{12}=a_{21}$,  Yagi \cite{Yag93} shows that $\vP(\vu)$, $u,v >0$, is positive definite when $8a_{11} > a_{21} $ and $ 8a_{22}> a_{12}$.  Thus \eqref{coef cond} not only extends Yagi's result in the case $a_{12}=a_{21}$,  but also extends to other cases when $a_{12}\not =a_{21}$ (in \eqref{coef cond b}).

\end{Rmk}

We thus assume the  coefficients satisfy  \eqref{coef cond}, which implies that

\begin{equation}
  \inner{\vP(\vu)\grad \vu}{\grad \vu} \ge \int_\Omega \big[ d_0 + \alpha(u+v) \big] \abs{\grad \vu}^2\,dx,
  \label{P positive definite 2}
\end{equation}
for all $ u,v > 0$.

Note that \eqref{P positive definite} implies that, for $u,v \ge 0$, $\vP(\vu)$ is invertible (as a $2\times 2$ matrix),
and that, pointwise (i.e. for a.e. $x\in \Omega$),
\begin{equation}\label{2.5b}
 \abs{\vP(\vu)^{-1}}_{\mathcal{L}(\mathbb{R}^2)} \le \f{1}{d_0+\alpha(u+v)}.
\end{equation}

We now complete the first a priori estimate \eqref{energy weak}. We observe that
\begin{align*}
 &\inner{\vq(\vu)}{\vu} = \int_\Omega \Big[( b_1u + c_1v) u^2 +(b_2u+c_2v)v^2 \Big] dx \ge 0,
 \\& \inner{\vl(\vu)}{\vu} = \int_\Omega (a_1 u^2 + a_2v^2)\,dx \le \int_\Omega \left( \f{b_1}{2}u^3+ \f{c_2}{2}v^3\right)dx + \mathcal{K}_1,
\end{align*}
where $\mathcal{K}_1$ is an absolute constant.
Finally \eqref{energy weak} yields
\begin{multline}
 \label{energy weak final}
 \f 12 \f{d}{dt}(\abs{u}^2 + \abs{v}^2) + \int_\Omega \Big[ d_0+ \alpha(u+v)\Big]\Big(\abs{\grad u}^2 + \abs{\grad v }^2\Big)\,dx
 \\+\f 12 \int_\Omega \Big[(b_1u+c_1v)u^2 + (b_2u+c_2v)v^2\Big] dx \le \mathcal{K}_1.
\end{multline}
Finally, \eqref{energy weak final} gives a  priori estimates for $u$ and $v$ in
 \begin{equation}
  \label{u v weak bounds}
  L^\infty(0,T;L^2(\Omega)) \text{ and } L^2(0,T;H^1(\Omega)),
 \end{equation}
and bounds of \begin{equation}
     \label{u half grad u bound}
     \sqrt{u}\grad u, \sqrt{u}\grad v, \sqrt{v} \grad u, \sqrt{v}\grad v \text{ in }L^2(0,T;L^2(\Omega)^2).
    \end{equation}

 We infer from \eqref{u half grad u bound} that $\sqrt{u}\grad u, \sqrt{v}\grad v \in L^2(L^2)$ and hence $\grad u^\f32 \in L^2(L^2)$. Now we observe that, for any $r \ge 1$
\begin{equation}
  \label{2.9a}
  \abs{\grad \phi}_{L^2} + \abs{\phi}_{L^{r}}.
\end{equation}
 is a norm on $H^1$, which is equivalent to the usual $H^1$ norm (thanks to the generalized Poincar\'{e} inequality, see \cite[Chapter 2]{Tem97}). Since
 $\dps \abs{u^\f32}_{L^\f43} = \left(\int_\Omega u^2\right)^\f34 = \abs{u}_{L^2}^\f32$,
 we see that

   \begin{equation}\label{2.8a}u^\f32 \in L^2(H^1).\end{equation}
 Then depending on dimension $d$,
  \begin{subequations}\label{2.8b}
 \begin{align}
  & u^\f32 \in L^2(L^\infty) \qquad d=1,
  \\& u^\f32 \in L^2(L^q) \qquad \forall q, d=2,
  \\&u^\f32 \in L^2(L^6) \qquad d=3,
  \\& u^\f32 \in L^2(L^4) \qquad d=4.
 \end{align}
 \end{subequations}

We also observe in space dimension $4$ (the most restrictive case), that \eqref{2.8b} implies that

\begin{equation}
  \label{2.10c}
  u \in L^\f32 (L^6)
\end{equation}
and since $u\in L^\infty(L^2)$ we can write
\[\int_\Omega u^4 = \int_\Omega u^3 u \le \left( \int_\Omega u^6\right)^\f12 \left(\int_\Omega u^2 \right)^\f12 = \abs{u}_{L^6}^\f32 \abs{u}_{L^2} \in L^1_t, \]
so that
\begin{equation}
  \label{2.10d}
  u \in L^4_t (L^4).
\end{equation}
We summarize the estimates in the following lemma
\begin{Lem} \label{lem: 2.1}
  We assume that $\vu_0 = (u_0,v_0) \in L^2(\Omega)^2$, $\vu_0\ge\vo$, that the coefficients $a_{ij}$ satisfy \eqref{coef cond}, $d\le 4$ and that $\vu = (u,v)$ is a smooth $\ge0$ solution of \eqref{SKT system}, \eqref{bry cond and init cond}. Then $u,v$ satisfy
  \eqref{u v weak bounds}, \eqref{u half grad u bound}, \eqref{2.8a} -- \eqref{2.10d}. Furthermore the norms in these spaces can be bounded by constants depending on the $L^2$ norms of $u_0, v_0$, on $T$ and on the coefficients.
\end{Lem}
\begin{Rmk} \label{rmk:2.11}
  \begin{enumerate}[ 1) ]
    \item We refrain from giving the explicit values of the constants bounding the norms mentioned in Lemma \ref{lem: 2.1}, as the corresponding explicit calculations will not bring any further information.
    \item We will see in Section \ref{sec:attractor} how one can improve Lemma \ref{lem: 2.1} and derive time uniform estimates for $t\ge 0$.
    \item The estimates in Lemma \ref{lem: 2.1} are those that we will use for the existence of weak solutions. In Section \ref{sec: second a priori est}, assuming more regularity on the initial data, we derive more a priori estimates leading to more regular solutions.
  \end{enumerate}
\end{Rmk}
\subsection{More a priori estimates\label{sec: second a priori est}}
In this section,  besides the assumption that $\vu_0=(u_0,v_0)\in
L^2(\Omega)^2$ as in the previous section, we assume further  that
\begin{equation}
  \label{u0 v0 strong}
  \grad \vp(\vu_0) \in L^2(\Omega)^4.
\end{equation}
We multiply \eqref{SKT system}$_1$ by $\partial_t p_1$ and \eqref{SKT system}$_2$ by $\partial_t p_2$ integrate and add
while noticing that
\[\begin{pmatrix}
   \partial_t p_1 \\ \partial_t p_2
  \end{pmatrix}
  = \vP(\vu) \begin{pmatrix}
              \partial_t u \\ \partial_t v
             \end{pmatrix},
\]
so that we have
\begin{multline}
 \label{energy strong}
 \inner{\vP(\vu) \partial_t \vu }{\partial_t \vu}  + \partial_t\Big( \abs{\grad p_1}^2_{L^2}+ \abs{\grad p_2}^2_{L^2}\Big)
 \\= \int_\Omega (q_1 \partial_t p_1 + q_2 \partial_t p_2)\, dx + \int_\Omega (\ell_1 \partial_t p_1 + \ell_2 \partial_t p_2) \,dx.
\end{multline}
By \eqref{P positive definite}, we have
\[ \inner{\vP(\vu) \partial_t \vu }{\partial_t \vu} \ge \int_\Omega \big[d_0 + \alpha(u+v)\big] \abs{\partial_t \vu}^2dx.\]
For the RHS of \eqref{energy strong}, we observe that
\[\partial_t p_1 = (d_1 +2a_{11}u+a_{12}v)\partial_t u  + a_{12}u \partial_t v \]
and the same for $\partial_t p_2$. Hence the RHS of \eqref{energy strong} is bounded by
\[\int_\Omega (u^3+v^3)\Big(\abs{\partial_t u} + \abs{\partial_t v}\Big)\,dx \le
\f{\alpha}{2} \int_\Omega (u+v)\Big( \abs{\partial_t u}^2 + \abs{\partial_t v}^2 \Big)\,dx + C \int_\Omega \Big( u^5 + v^5 \Big)\,dx.\]
In the end
\begin{equation}
 \label{energy strong 1}
 \int_\Omega\Big[ d_0 + \alpha(u+v) \Big]\abs{\partial_t \vu}^2 dx + \partial_t \Big(\abs{\grad p_1}^2_{L^2}+ \abs{\grad p_2}^2_{L^2}\Big)
 \le C \int_\Omega \Big(u^5+v^5\Big)\,dx.
\end{equation}
We infer from \eqref{u v weak bounds} that \begin{align*}
 p_1  = d_1  u + a_{11}u^2 + a_{12}uv  \in L^\infty(L^1).
\end{align*}

 We know from the Sobolev inclusion that
 \begin{align*}\
& H^1(\Omega) \subset L^{\f{2d}{d-2}}(\Omega)
\end{align*}
with continuous embedding.

We know from \eqref{2.9a} that $\abs{\phi}_{L^1} + \abs{\grad \phi}_{L^2}$ is a norm on $H^1(\Omega)$ equivalent to the usual norm. Furthermore, because $u,v\ge 0$, for $d=4$, we have
\begin{equation}\label{u2 sobolev}\abs{u^2}_{L^4} \le a_{11}\abs{p_1}_{L^4} \le C \Big(\abs{\grad p_1}_{L^2}
+\abs{p_1}_{L^1} \Big).\end{equation}
Using H\"{o}lder's inequality, we find
\begin{multline*}\int_\Omega u^5 \,dx = \int_\Omega u \cdot u^4 \,dx\le   \ubrace{\left(\int_\Omega u^2 \,dx \right)^\f12}{L^\infty_t} \left(\int_\Omega (u^4)^2 \right)^\f12 \le  C \left(\int_\Omega (u^2)^4 \right)^{\f14 \times 2}
\\ =  C \abs{ u^2}_{L^4} ^2 \le (\text{using \eqref{u2 sobolev}}) \le C \Big(\abs{\grad p_1}_{L^2}^2
+\abs{p_1}_{L^1}^2 \Big). \end{multline*}
From \eqref{u v weak bounds}, we see that $\dps \abs{p_1}_{L^1}^2 = \left(\int_\Omega (d_1 + a_{11} u+ a_{12}v) u \right)^2 \le C$. Thus
\begin{subequations}
 \begin{equation}\label{u5 bound}
\int_\Omega u^5 \,dx \le C \Big(\abs{\grad p_1}_{L^2}^2 + 1\Big),
\end{equation}
and similarly
\begin{equation}\label{v5 bound}
\int_\Omega v^5 \,dx \le C \Big(\abs{\grad p_2}_{L^2}^2 + 1\Big).
\end{equation}
\end{subequations}
Combing \eqref{energy strong 1}, \eqref{u5 bound} and \eqref{v5 bound}, we have for the least easy case $d=4$:
\begin{multline}
\label{energy strong 2}
\int_\Omega\Big[ d_0 + \alpha(u+v) \Big]\abs{\partial_t \vu}^2 dx + \partial_t \Big(\abs{\grad p_1}^2_{L^2}+ \abs{\grad p_2}^2_{L^2}\Big)
\\\le C  \Big(\abs{\grad p_1}^2_{L^2}+ \abs{\grad p_2}^2_{L^2} +1 \Big)  .
\end{multline}

Using Gronwall's inequality for the function $\mc{Y}(t) = \abs{\grad p_1(t)}^2_{L^2}+ \abs{\grad p_2(t)}^2_{L^2}$, we conclude that $ \mc{Y}(t) $ is bounded on $[0,T]$ by a constant, independent of $t$ and dependent on $\abs{\grad \vp (\vu_0)}_{L^2}$ and $T$. By using the bound of $\mc{Y}(t)$ on the RHS  of \eqref{energy strong 2}, we conclude that
\begin{equation}
 \label{u v strong bounds}
 \begin{cases}
 &\partial_t u,\partial_t v,\sqrt{u+v}\left( \abs{\partial_t u }+ \abs{\partial_t v}\right) \in L^2(L^2),\\
 &\grad p_1,\grad p_2 \in L^\infty(L^2) \quad (p_1,p_2 \in L^\infty(H^1)).
 \end{cases}
\end{equation}
Noting that, by \eqref{P}-\eqref{P jac}, $\grad \vp = \vP(\vu) \grad \vu$, we have
\begin{equation}
  \label{grad u grad p}
  \grad \vu =\vP(\vu)^{-1} \grad \vp.
\end{equation}
Then in view of \eqref{2.5b} and \eqref{u v strong bounds}, we have
\begin{equation}
 \label{2.12a}
 \grad u, \grad v \in L^\infty(L^2).
\end{equation}

\begin{Rmk} We deduce from \eqref{u v strong bounds}$_2$ that $\grad p_i \in L^\infty(L^2)$
and we want to justify that $p_i\in L^\infty(H^1)$ through an equivalent norm. Indeed,
we know that  $u,v$ are in $L^\infty(L^2)$ and thus $p_i  \in L^\infty(L^1)$ (since $p_i \approx u^2+v^2$).
By the generalized Poincar\'{e} inequality \eqref{2.9a} we see that $ \abs{\grad \phi}_{L^2} + \abs{\phi}_{L^1}$ is a norm on $H^1$ which is equivalent to the usual norm. This says that $p_i\in L^\infty(H^1)$.
\end{Rmk}
  From \eqref{2.10d} we infer that
  \begin{equation}\label{2.14b}
    \vq(\vu),\vl(\vu) \in L^2(0,T;L^2(\Omega)^2).
  \end{equation}

Then returning to \eqref{SKT system} and using \eqref{u v strong bounds}$_1$, we see that
\begin{equation}
  \label{2.14c} \Delta p_1,\Delta p_2 \in L^2(0,T;L^2(\Omega)),
\end{equation}
and we have a priori bounds for $\Delta p_1, \Delta p_2$ in these spaces.

\begin{Lem} \label{lem: 2.2}
We assume that $\vu_0 \in L^2(\Omega)^2$, $\vu_0\ge \vo$ and $\grad \vp(\vu_0)
\in L^2(\Omega)^4$, that the coefficients $a_{ij}$ satisfy \eqref{coef
cond}, $d\le 4$ and that $\vu = (u,v)$ is a smooth $\ge 0$ solution of
\eqref{SKT system}, \eqref{bry cond and init cond}. Then $u,v$ satisfy
  \eqref{u v strong bounds}, \eqref{2.12a}--\eqref{2.14c}. Furthermore the norms of $\vu$ and $\vp$ in these spaces can be bounded by constants depending on the norms of $\vu_0$ and $\grad \vp(\vu_0)$ in $L^2$, on $T$ and on the coefficients.
\end{Lem}

\subsection{Positivity and maximum principle\label{sec: max principle}}
To deal with the positivity of the solutions, we will consider the following auxiliary systems

\begin{equation}
 \label{SKT aux 1}
 \partial_t \vu - \grad \cdot \Big(\vP(\vu^+) \grad \vu\Big) + \vq(\vu^+) = \vl(\vu^+),
\end{equation}
where $\vu^+ = (u^+,v^+)$ and
\begin{equation}
 \label{SKT aux 2}
 \partial_t \vu - \grad \cdot \Big(\vP(\vlambda_M(\vu)) \grad \vu\Big) + \vq(\vlambda_M(\vu)) = \vl(\vlambda_M(\vu)),
\end{equation}
where $M>0$ and \begin{equation}
      \label{lambda M}
      \vlambda_M(\vu) =( \lambda_M(u) , \lambda_M(v))
      \text{ with }
      \begin{cases}
                        & \lambda_M(u) = u^+  \text{ for }u\le M \text{ and } M \text{ for } u>M,\\
                        &\lambda_M(v) = v^+  \text{ for }v\le M \text{ and } M \text{ for } v>M.
                       \end{cases}
     \end{equation}
Note that if $\vu$ is a smooth solution of either system then $\vu \ge \vo$ $(u\ge 0, v\ge0)$ so that $\vu^+ =\vu$.

Indeed taking the scalar product in $L^2$ of \eqref{SKT aux 1} with $-\vu^-$, we find
\begin{itemize}
 \item  $\dps -\inner{\partial_t \vu}{\vu^-} =\f12 \f{d}{dt}\abs{\vu^-}^2$,
 \item  $\dps \aligned[t]
            &-\inner{\grad \cdot \Big( \vP(\vu^+) \grad \vu \Big)}{\grad \vu^-} =
 - \inner{\vP(\vu^+)\grad \vu}{\grad \vu^-}\\
 & \qquad\qquad\qquad= -\int_\Omega \Big[\big( d_1+2a_{11}u^++a_{12}v^+ \big) \ubrace{\grad u \grad u^-}{\le 0} +
 a_{12}  \ubrace{u^+ \grad v
  \grad u^-}{=0} \Big]\,dx
 \\&\qquad \qquad\qquad\qquad\qquad\qquad -\int_\Omega \Big[ a_{21}  \ubrace{v^+  \grad u
 \grad v^-}{=0}
 + \big( d_1+a_{21}u^++2a_{22}v^+ \big) \ubrace{\grad v \grad v^-}{\le 0}\Big]\,dx.
               \endaligned$

\end{itemize}
             With $\dps \inner{\grad u}{\grad u^-} = -\abs{\grad u^-}^2$, and the same for $v$, we see that
             $-\inner{\grad \cdot \Big(\vP(\vu^+)\grad \vu\Big) }{u^-} \ge 0$.

             Similarly
\begin{itemize}
             \item  $\dps \aligned[t] & -\int_\Omega \Big[ q_1(\vu^+)u^- + q_2 (\vu^+)v^-\Big]\,dx
             \\& = \int_\Omega \Big[ \big( b_1u^++c_1v^+\big)u^+ u^- + \big( b_2u^++c_2v^+\big)v^+ v^- \Big]\,dx = 0,
             \endaligned$
             \item $\dps -\inner{\ell_1(u^+)}{u^-} -\inner{\ell_2(v^+)}{v^-} = -\int_\Omega \Big[a_1 u^+ u^- + a_2 v^+ v^-\Big] \,dx =0$.
\end{itemize}
Finally
\begin{equation}
 \label{u- der neg}
 \f{d}{dt}\abs{\vu^-}^2 \le 0
\end{equation}
so that $\vu^-(t) = 0$ for all time if $u_0 \ge 0, v_0 \ge 0$.

Then multiplying \eqref{SKT aux 2} by $-\vu^-$, we find
\begin{itemize}
 \item $\dps \aligned[t]
 & \inner{\grad \cdot \Big(\vP(\vlambda_M(\vu)) \grad \vu \Big)}{\vu^-} = - \inner{ \Big(\vP(\vlambda_M(\vu)) \grad \vu \Big) }{\grad \vu^-}\\&
 \qquad\quad= -\int_\Omega \Big[\big( d_1+2a_{11}\lambda_M(u)+a_{12}\lambda_M(v) \big) \ubrace{\grad u \grad u^-}{\le 0}
 +a_{12}  \ubrace{\lambda_M(u)
  \grad v \grad u^-} {=0} \Big]
 \\&\qquad \qquad\qquad\qquad-\int_\Omega \Big[ a_{21}
    \ubrace{\lambda_M(v)
 \grad u \grad v^-}{=0}
 + \big( d_1+a_{21}\lambda_M(u)+2a_{22}\lambda_M(v) \big) \ubrace{\grad v \grad v^-}{\le 0}\Big].
 \\& \qquad\quad \ge 0.
\endaligned $
 \end{itemize}
Also
\begin{itemize}
             \item  $\dps \aligned[t] & -\int_\Omega\Big[ q_1(\vlambda_M(\vu))u^- + q_2 (\vlambda_M(\vu))v^- \Big]\,dx
             \\& = \int_\Omega \Big[\big( b_1\lambda_M(u)+c_1\lambda_M(v)\big)\lambda_M(u) u^- + \big( b_2\lambda_M(u)+c_2\lambda_M(v)\big)\lambda_M(v) v^- \Big]\,dx= 0,
             \endaligned$
             \item $\dps -\inner{\ell_1(\lambda_M(u))}{u^-} -\inner{\ell_2(\lambda_M(v))}{v^-} = -\int_\Omega \Big[ a_1 \lambda_M(u) u^- + a_2 \lambda_M(v) v^- \Big]\,dx=0$,
\end{itemize}
so that \eqref{u- der neg} holds again leading to $u(t)\ge 0, v(t)\ge 0$ for all $t.$

\paragraph{Orientation.} It is not easy to take advantage of the above a priori estimates (in particular those of Section \ref{sec: first apriori est}),
and to construct positive solutions of the SKT equations satisfying these a priori estimates. We will deal with this issue in Section \ref{sec: FD}
by using finite differences in time, together with the truncation operator $\lambda_M$.

\section{Finite differences in time \label{sec: FD}}
\subsection{Finite differences for the truncated SKT equations}
Let $T>0$  be fixed. Consider two numbers $N,M>0$ fixed for the moment but which will eventually go to infinity; $N$ is an integer and $k=\Delta t
=T/N$ is the time step.

We consider the finite difference scheme:
\begin{equation}
 \label{SKT FD}
 \begin{cases}
  & \dps \f{u_M^m-u_M^{m-1}}{k} -\grad \cdot \Big( p_{11}^M(u_M^m,v_M^m) \grad u_M^n +  p_{12}^M(u_M^m,v_M^m) \grad v_M^n \Big)\qquad \qquad \qquad \qquad
  \\ & \hfill + q_1^M(u_M^m,v_M^m)-\ell_1^M(u_M^m) = 0
  \\ & \dps \f{v_M^m-v_M^{m-1}}{k} -\grad \cdot \Big( p_{21}^M(u_M^m,v_M^m) \grad u_M^n +  p_{22}^M(u_M^m,v_M^m) \grad v_M^n \Big)\qquad \qquad \qquad \qquad
  \\ & \hfill + q_2^M(u_M^m,v_M^m)-\ell_2^M(v_M^m) = 0,
 \end{cases}
\end{equation}
with boundary conditions
\[u_M^m=v_M^m=0 \text{ on }\partial \Omega \quad (\text{or }\f{\partial u_M^m}{\partial n} = \f{\partial u_M^m}{\partial n} =0 \text{ on }\partial \Omega)\]
and initial conditions
\[u_M^0=u_0, v_M^0=v_0.\]
We have set
\begin{subequations}
\begin{equation}
\label{pMi}
 p_{ij}^M(u_M^m,v_M^m) = p_{ij}(\lambda_M(u_M^m),\lambda_M(v_M^m)), \quad i,j=1,2.
\end{equation}
\begin{equation}
\label{lMi}
\ell_1^M(u_M^m) = \ell_1(\lambda_M(u_M^m)),\quad\ell_2^M(v_M^m)=\ell_1(\lambda_M(v_M^m)),
\end{equation}
However, for the $q_i$ we write
\begin{equation}
 \label{qMi}
 \begin{cases}
 & q_1^M(u_M^m,v_M^m) = (b_1 \abs{u_M^m}+c_1 \lambda_M(v_M^m) )u_M^m,\\
 & q_2^M(u_M^m,v_M^m) = (b_2 \lambda_M(u_M^m)+c_2 \abs{v_M^m} )v_M^m,
 \end{cases}
\end{equation}
\end{subequations}

\paragraph{Alternatively in variational form:} $u_M^m, v_M^m \in V$  ($=H^1(\Omega)^2 \text{ or } H^1_0(\Omega)^2$ depending on the b.c.)  satisfying
\begin{equation}\label{SKT FD var}
 \begin{cases}
   & \inner{u_M^m}{\bar{u}} + k \inner{p_{11}^M(u_M^m,v_M^m)  \grad u_M^m }{\grad \bar{u}} + k \inner{p_{12}(u_M^m,v_M^m) \grad v_M^m }{\grad \bar{u}}
   \qquad \qquad \qquad
   \\& \hfill + k\inner{q_1^M(u_M^m,v_M^m) }{\bar{u}} - k \inner{\ell_1^M(u_M^m) }{\bar{u}} = \inner{u_M^{m-1}}{\bar{u}},
   \\
    & \inner{v_M^m}{\bar{v}} + k \inner{p_{21}^M(u_M^m,v_M^m)  \grad u_M^m }{\grad \bar{v}} + k \inner{p_{22}(u_M^m,v_M^m) \grad v_M^m }{\grad \bar{v}}
   \qquad \qquad \qquad
   \\& \hfill + k\inner{q_2^M(u_M^m,v_M^m) }{\bar{v} }- k \inner{\ell_2^M(v_M^m) }{\bar{v}} = \inner{v_M^{m-1}}{\bar{v}},
 \end{cases}
\end{equation}
for every $\bar{u},\bar{v} \in V$.

For the sake of simplicity we temporarily drop the lower index $M$ and write $u^m,v^m$ instead of $u_M^m, v_M^m$.

Because $\lambda_M$ is a bounded function, the proof of existence of $u^m,v^m$ follows very closely the proof of existence of solutions for the stationary Navier Stokes equations
\cite{Tem01} after we notice the following ``coercivity'' properties obtained by setting $u=u^m$ and $v=v^m$ in the left hand sides of \eqref{SKT FD var}$_1$ and \eqref{SKT FD var}$_2$ and adding these
equations
\begin{multline} \label{energy fd 1}
\abs{u^m}^2 + k\inner{p_{11}^M(u^m,v^m) \grad u^m}{\grad u^m}
 + k\inner{p_{12}^M(u^m,v^m) \grad v^m}{\grad u^m}
 + k\inner{q_1^M(u^m,v^m) }{u^m}
 \\ - k\inner{\ell_1^M(u^m)}{u^m}
 +\abs{v^m}^2 + k\inner{p_{21}^M(u^m,v^m) \grad u^m}{\grad v^m} \hspace{3cm}
 \\+ k\inner{p_{22}^M(u^m,v^m) \grad v^m}{\grad v^m}
  + k\inner{q_2^M(u^m,v^m) }{u^m}
\\ - k\inner{\ell_2^M(v^m)}{v^m}=   \inner{u^{m-1}}{u^m}-\inner{v^{m-1}}{v^m}.
\end{multline}

Repeating the calculations in \eqref{P positive definite} under the assumptions  in \eqref{coef cond}, we see
that the sum of the $p_{ij}$ terms is bounded from below by
\begin{equation}\label{pij bound fd}
 \int_\Omega \big[ d_0+\alpha(\lambda_M(u^m) + \lambda_M(v^m))\big] \Big[ \abs{\grad u^m}^2 + \abs{\grad v^m}^2 \Big] \,dx.
\end{equation}
The $q_i$ terms give the lower bound
\begin{multline} \label{qi bound fd}
 \int_\Omega \Big[ b_1\abs{u^m}^3 + c_1\lambda_M(v^m)\abs{u^m}^2 + b_2 \lambda_M(u^m)\abs{v^m}^2+c_2  \abs{v^m}^3 \Big]dx
 \\\ge \int_\Omega \Big( b_1\abs{u^m}^3+c_2\abs{v^m}^3\Big)dx.
\end{multline}

We easily see that the $\ell_i$ terms are bounded from below by
\begin{equation} \label{li bound fd}
 -\int_\Omega \Big[ a_1 \abs{u^m}^2 + a_2 \abs{v^m}^2 \Big]dx
 \ge -\f12 \int_\Omega \Big[ b_1\abs{u^m}^3 + c_2 \abs{v^m}^3\Big]dx -\mathcal{K}_1,
\end{equation}
where $\mathcal{K}_1$ is an absolute constant.

Hence the expression \eqref{energy fd 1} is bounded from below pointwise a.e. by
\begin{multline} \label{energy lhs bound fd}
 \abs{u^m}^2 + \abs{v^m}^2 +\big[ d_0+\alpha(\lambda_M(u^m) + \lambda_M(v^m))\big]\Big( \abs{\grad u^m}^2 + \abs{\grad v^m}^2\Big)
 \\+ \f{b_1}{2} \abs{u^m}^3 +\f{c_1}{2}  \abs{v^m}^3 - \mathcal{K}_1,
\end{multline}
which guarantees coercivity in (at least) $V$ for \eqref{SKT FD var}.

As we said, by implementation of a Galerkin method, as for the stationary Navier-Stokes equations, we obtain the existence of $(u^m,v^m) \in V$ solutions of
\eqref{SKT FD var}.

Then, as for \eqref{u- der neg}, we show recursively, starting from $u_0,v_0\ge 0$, that the $u^m,v^m$ are $\ge 0$. Indeed, e.g. for $u^m$, we replace
$\bar{u}$ by $-(u^m)^-\in H^1(\Omega)$ (or $H^1_0(\Omega)$) in \eqref{SKT FD var}$_1$. This gives
\begin{multline}\label{energy um-}
 \abs{(u^m)^-}^2 - k \int_\Omega \Big(d_1+2a_{11}\lambda_M(u^m)+a_{12}\lambda_M(v^m) \Big) \ubrace{\grad u^m \grad (u^m)^-}{\le 0}
 \\ -k \int_\Omega a_{12}
 \ubrace{\lambda_M(u^m)
 \grad v^m
 \grad(u^m)^-}{=0}
 -\int_\Omega\Big( b_1\abs{u^m} + c_1\lambda_M(v^m)\Big)\ubrace{u^m (u^m)^-}{\le0} \,dx
\\+\int_\Omega a_1 \ubrace{\lambda_M(u^m)(u^m)^-}{=0}\,dx = - \inner{u^{m-1}}{(u^m)^-}.
\end{multline}

The RHS of \eqref{energy um-} is $\le 0$, since, by the induction assumption, $u^{m-1}\ge 0$. Hence
\[(u^m)^-= 0,\]
and $u^m=u_M^m\ge 0$. We proceed similarly for $v^m=v_M^m$.

Having shown that $u^m=u_M^m, v^m=v_M^m$ are $\ge 0$, we can drop the absolute values that we have introduced in the definition
of $q_1^M$ and $q_2^m$ in \eqref{qMi}.

\begin{Rmk}
 Now we remember that $u^m,v^m$ actually  depend on $M$, $u^m=u_M^m, v^m=v_M^m$, with $u_M^m,v_M^m\ge 0$, $u_M^m,v_M^m\in
 H^1(\Omega) \;(\cap L^3(\Omega)) $, and we want to let $M\rightarrow \infty$, to obtain a finite difference approximation for the SKT equations themselves.
\end{Rmk}

\subsection{Finite differences for the SKT equations}
We now want to let $M\rightarrow \infty$ in \eqref{SKT FD var} to obtain a solution to the finite difference scheme for the SKT equations
themselves. For the moment $M$ is still fixed.

Considering the solutions $u_M^m,v_M^m$ of \eqref{SKT FD var}, we write $\bar{u}=2u_M^m$ in \eqref{SKT FD var}$_1$ and $\bar{v}=2v_M^m$
in \eqref{SKT FD var}$_2$ and add these equations. This gives

\begin{equation}\label{energy fd 2}
 2\inner{u_M^m-u_M^{m-1}}{u_M^m} + 2\inner{v_M^m-v_M^{m-1}}{v_M^m} + 2\mathcal{L}_M^m =0,
\end{equation}
where $\mc{L}_M^m$ is the expression in the left-hand-side of \eqref{energy fd 1}  (less $\abs{u_M^m}^2+ \abs{v_M^m}^2$).

With the classical relation $2\inner{a-b}{b}=\abs{a}^2-\abs{b}^2+\abs{a-b}^2$, \eqref{energy fd 2} yields
\begin{equation} \label{energy fd 3}
 \abs{u_M^m}^2 + \abs{v_M^m}^2- \abs{u_M^{m-1}}^2 - \abs{v_M^{m-1}}^2+  \abs{u_M^m-u_M^{m-1}}^2  + \abs{v_M^m-v_M^{m-1}}^2
 +2\mc{L}_M^m =0.
\end{equation}
We then infer from \eqref{energy fd 3} and the minorations \eqref{pij bound fd} and \eqref{energy lhs bound fd} that
\begin{multline}\label{energy fd 4}
 \abs{u_M^m}^2 + \abs{v_M^m}^2+  \abs{u_M^m-u_M^{m-1}}^2  + \abs{v_M^m-v_M^{m-1}}^2
 \\ + 2k\int_\Omega \Big(d_0 + \alpha(\lambda_M(u_M^m) +\lambda_M(v_M^m) \Big)\Big( \abs{\grad u_M^m}^2 + \abs{\grad v_M^m}^2 \Big)\,dx
\\ + kb_1\abs{u_M^m}^3_{L^3}  +kc_1 \abs{v_M^m}^3_{L^3} \le \mc{K}_1 + \abs{u_M^{m-1}}^2+ \abs{v_M^{m-1}}^2,
\end{multline}
for $m=1,\dots,N$. By addition and iteration we obtain that
\begin{subequations}
\label{uMm bounds}
\begin{align}
&\label{uMm bounds1}\abs{u_M^m}^2 +\abs{v_M^m}^2 \le \mc{K}_2,\\
& \label{uMm bounds2}\sum_{m=1}^N \abs{u_M^m-u_M^{m-1}}^2 + \abs{v_M^m-v_M^{m-1}}^2 \le \mc{K}_2,\\
& \label{uMm bounds3}k\sum_{m=1}^N \Big(\abs{\grad u_M^m}^2 + \abs{\grad v_M^m}^2 \Big) \le \mc{K}_2,\\
& \label{uMm bounds4}k\sum_{m=1}^N \int_\Omega\big(\lambda_M(u_M^m)+\lambda_M(v_M^m)\big)\Big(\abs{\grad u_M^m}^2 + \abs{\grad v_M^m}^2 \Big)dx \le \mc{K}_2,\\
&\label{uMm bounds5}k\sum_{m=1}^N \Big[ \abs{u_M^m}^3_{L^3}  + \abs{v_M^m}^3_{L^3} \Big]\le \mc{K}_2,
 \end{align}
\end{subequations}
where $\mc{K}_2$ is a constant independent of $M$ and $N$ (but which depends on $\abs{u_0}^2 +\abs{v_0}^2$ and the other data).

The other estimates will be used later on but for the moment, for $N$ fixed, we infer from \eqref{uMm bounds1} and \eqref{uMm bounds3}
that the $u_M^m$ and $v_M^m$ are bounded in $H^1(\Omega) \cap L^3(\Omega)$ independently of $M$, $m=1,\dots,N$. Hence
by a finite number of extraction of subsequence we see that for $M\rightarrow \infty$.
\begin{equation}
 \label{uMm vMm to um vm}
 u_M^m \rightarrow u^m, \; v_M^m \rightarrow v^m,
\end{equation}
weakly in $H^1(\Omega)$ and strongly in $L^2(\Omega) \cap L^3(\Omega) $ in dimensions $d \le 4$ for some $u^m,v^m \in H^1(\Omega)$ which
are $\ge 0$ like the $u_M^m,v_M^m$. Also by additional extraction of subsequences, the convergences \eqref{uMm vMm to um vm} hold almost everywhere in $\Omega$.


It is relatively easy to pass to the limit in the relation \eqref{SKT FD var}, with $(\bar{u},\bar{v}) \in V \cap \mc{C}(\wbar{\Omega})^2$ using the convergences \eqref{uMm vMm to um vm}. We obtain that for each $m=1,\dots,N$
\begin{equation}
  \label{SKT FD var limit}
  \begin{cases}
  & \dps   \inner{u^m}{\bar{u}} + k \inner{p_{11}(u^m,v^m)\grad u^m}{ \grad \bar{u}} + k \inner{p_{12}(v^m,u^m) \grad v^m}{\grad \bar{u}} \hspace{2cm}
  \\ & \hfill + k\inner{q_1(u^m,v^m)}{\bar{u}} - k \inner{\ell_1(u^m)}{\bar{u}} = \inner{u^{m-1}}{\bar{u}},
  \\  & \dps   \inner{v^m}{\bar{v}} + k \inner{p_{21}(u^m,v^m)\grad v^m}{ \grad \bar{v}} + k \inner{p_{22}(v^m,u^m) \grad v^m}{\grad \bar{v}}
  \\ & \hfill + k\inner{q_2(u^m,v^m)}{\bar{v}} - k \inner{\ell_2(v^m)}{\bar{v}} = \inner{v^{m-1}}{\bar{v}},
  \end{cases}
\end{equation}
for any $(\bar{u},\bar{v}) \in V\cap \mc{C}(\wbar{\Omega})^2$. We observe for this purpose that
\begin{align}
  \label{lambdaMUMm convergence}
  &\lambda_M(u_M^m) \rightarrow (u^m)^+ = u^m, \,\lambda_M(v_M^m) \rightarrow (v^m)^+ = v^m.
\end{align}
a.e. and in $L^2(\Omega)$.

\paragraph{Additional information.}  Now we want to pay more attention to the terms $p_{ij}(u^m,v^m)\grad u^m$, $ p_{ij}(u^m,v^m)\grad v^m$ and to show that \eqref{SKT FD var limit} is valid in fact for any $\bar{u},\bar{v} \in V$ (and not only
$\bar{u},\bar{v} \in V\cap \mc{C}(\wbar{\Omega})^2$).

We start from the estimate \eqref{uMm bounds4}, and consider the expressions
\[\sqrt{\lambda_M(u_M^m)}\grad u_M^m, \sqrt{\lambda_M(u_M^m)}\grad v_M^m,  \sqrt{\lambda_M(v_M^m)}\grad u_M^m, \sqrt{\lambda_M(v_M^m)}\grad v_M^m.\]

They are, each, bounded in $L^2(\Omega)$, and they contain each a subsequence which converges weakly in $L^2(\Omega)$ ($m$ fixed, $M\rightarrow \infty$, finite extraction of subsequences). Due to \eqref{uMm vMm to um vm} and the a.e. convergences, their respective limits are
\[\sqrt{(u^m)^+}\grad u^m, \sqrt{(u^m)^+}\grad v^m,  \sqrt{(v^m)^+}\grad u^m, \sqrt{(v^m)^+}\grad v^m.\]
Passing to the lower limit in \eqref{uMm bounds4}, we see that
\begin{equation}
  \label{3.17}
  k\sum_{m=1}^N (u^m+v^m)\Big( \abs{\grad  u^m}^2+\abs{\grad v^m}^2\Big) \le \mc{K}_2.
\end{equation}
This estimate will eventually lead to the results that
$\sqrt{u}\grad u, \sqrt{u}\grad v,  \sqrt{v}\grad u, \sqrt{v}\grad v$ belong to $L^2(0,T;L^2(\Omega))$ as $N\rightarrow \infty$.

For the moment, with $N$ fixed, \eqref{lambdaMUMm convergence} and \eqref{uMm bounds1} imply that $\sqrt{u^m}\grad u^m$ and $\sqrt{v^m}\grad v^m \in L^2(\Omega)$. Hence $(u^m)^\f32$ and $(v^m)^\f32$ belong to $H^1(\Omega)$ and by Sobolev inclusion, say in dimension $d=4$ (the worse case) $(u^m)^\f32, (v^m)^\f32 \in L^4(\Omega)$, that is $u^m,v^m \in L^6(\Omega)$. Hence the expressions $u^m\grad u^m,$ $ u^m\grad v^m, $ $v^m\grad u^m, $ $v^m\grad v^m$ all belong to $L^\f{12}{7}(\Omega)$ as the product of an $L^2(\Omega)$ function with an $L^{\f{12}{5}}(\Omega)$ function; e.g.
\begin{equation}\label{3.17a}u^m\grad u^m = \sqrt{u^m} \sqrt{u^m}\grad u^m \in L^\f{12}{7}(\Omega).\end{equation}
Hence the equations \eqref{SKT FD var limit} are now valid for $(\bar{u},\bar{v})$ belonging to $V$ with $\grad \bar{u}, \grad \bar{v} \in L^{\f{12}{5}}(\Omega)^2$.

More generally passing to the lower limit $M\rightarrow \infty$ in \eqref{uMm bounds1}-\eqref{uMm bounds5} we see that for $m=0,\dots,N$
\begin{subequations}
\label{um bounds}
\begin{align}
&\label{um bounds1}\abs{u^m}^2 +\abs{v^m}^2 \le \mc{K}_2,\\
& \label{um bounds2}\sum_{m=1}^N \abs{u^m-u^{m-1}}^2 + \abs{v^m-v^{m-1}}^2 \le \mc{K}_2,\\
& \label{um bounds3}k\sum_{m=1}^N \Big(\abs{\grad u^m}^2 + \abs{\grad v^m}^2 \Big) \le \mc{K}_2,\\
& \label{um bounds4}k\sum_{m=1}^N \int_\Omega\big(u^m+v^m\big)\Big(\abs{\grad u^m}^2 + \abs{\grad v^m}^2 \Big)dx \le \mc{K}_2,\\
&\label{um bounds5}k\sum_{m=1}^N \abs{u^m}^3_{L^3}+ \abs{v^m}^3_{L^3} \le \mc{K}_2,
 \end{align}
\end{subequations}
with the same constant $\mc{K}_2$ independent of $N$.

We conclude our a priori estimates for the finite solutions $\vu^m=(u^m,v^m)$
\begin{Lem} \label{lem:dis apriori 1}Suppose that $\vu^0=\vu_0  \in L^2(\Omega)^2$, $\vu_0 \ge \vo$ and that \eqref{coef cond} holds. Then the bounds in \eqref{um bounds} hold true for the solutions $\vu^m$ of \eqref{SKT FD var limit}.
\end{Lem}
\begin{Rmk}
Later on we will extend this class of functions $(\bar{u},\bar{v}) $ such that \eqref{SKT FD var limit} holds and study the dependance in $t$. For the moment we aim to derive for the $u^m,v^m$ the discrete analogue of the a priori estimates \eqref{u v strong bounds}.
\end{Rmk}
\subsection{More a priori estimates}
Referring to the initial form \eqref{SKT system} of the SKT equations,
we rewrite \eqref{SKT FD var limit} as
\begin{equation}
\label{SKT fd alternate}
\begin{cases}
  &\dps \f{u^m-u^{m-1}}{k} - \Delta p_1(u^m,v^m) + q_1(u^m,v^m) -\ell_1(u^m) =0,
\\& \dps \f{v^m-v^{m-1}}{k} - \Delta p_2(u^m,v^m) + q_2(u^m,v^m) -\ell_2(v^m) =0,
  \end{cases}
\end{equation}
with, as in \eqref{pi qi li},
\begin{equation} \label{pi}
\begin{cases}
 & p_1(u^m,v^m) = (d_1+a_{11}u^m+a_{12}v^m)u^m,
 \\& p_2(u^m,v^m) = (d_2+a_{21}u^m+a_{22}v^m)v^m.
 \end{cases}
 \end{equation}

We have observed that, for fixed $N$, each $u^m,v^m$ belongs to $L^6(\Omega)$ ($d\le 4$). Hence $p_1(u^m,v^m),$ $p_2(u^m,v^m)$, $q_1(u^m,v^m),q_2(u^m,v^m), $
$\ell_1(u^m), \ell_2(v^m)$ belong to $L^{3}(\Omega)$.
We also observe that $\grad p_1(u^m,v^m) $ $= (d_1 + 2 a_{11}u^m +a_{12}v^m)\grad u^m+ a_{12}u^m\grad v^m$, and, in view of \eqref{um bounds3}, \eqref{um bounds4}, and the similar relations, $\grad p_1(u^m,v^m) \in L^\f{12}{7}(\Omega)$, and the same is true for $\grad p_2(u^m,v^m )$.
Furthermore, $p_1^m=p_1(u^m,v^m)$, $p_2^m=p_2(u^m,v^m ) $ satisfy the same
boundary conditions as $u^m,v^m$, and hence, in view of \eqref{lambdaMUMm convergence}, $p_1^m,p_2^m$ belong to $W^{2,12/7}(\Omega)$
for each fixed $m$ for $N$ fixed.
By bootstrapping, $\grad p_1^m,\grad p_2^m \in W^{1,12/7}(\Omega)^2 \subset L^3(\Omega)^2$ in space
dimensions $d\le 4$, so that $\grad p_1^m, \grad p_2^m \in L^2(\Omega)^2$, and $p_1^m,p_2^m \in H^2(\Omega)$,
for $m=0,1,\dots N$, $N$ fixed.

We have the following a priori estimates
\begin{Lem}\label{lem: dis apriori 2} Suppose that $\vu^0=\vu_0 \in L^2(\Omega)^2,\vu_0\ge \vo$, $\grad \vp(\vu_0) \in L^2(\Omega)^4$ and that the conditions \eqref{coef cond} are satisfied. Then the following bounds independent of $N$ hold true
  \begin{subequations} \label{strong bounds}
    \begin{align}
      \label{qn strong est}
      &\abs{\grad \vp(\vu^m)}_{L^2(\Omega)^4} \le c,  m=0,\dots,N,\\
        & \label{un strong est} k \sum_{m=1}^N \int_\Omega \Big[ d_0 + \f{\alpha}{4}(u^{m-1}+v^{m-1}+u^m+v^m)\Big]\abs{\f{\vu^m-\vu^{m-1}}{k}}^2\,dx \le c
    ,
    \end{align}
    where the constants $c$ are independent of $k$.
  \end{subequations}
\end{Lem}
\begin{proof}
  We take the scalar product in $L^2(\Omega)$ of \eqref{SKT fd alternate}$_1$ with $2k(p_1(\vu^m) -p_1(\vu^{m-1}))$ and \eqref{SKT fd alternate}$_2$ by $2k(p_2(\vu^m) -p_1(\vu^{m-1}))$. It is clear that
  \begin{subequations}\label{3.18}
  \begin{equation}
    \label{3.18a}
    -2k\inner{\Delta p^m_1}{p_1^m-p_1^{m-1}} = k\abs{\grad p_1^m}^2 - k\abs{\grad p_1^{m-1}}^2 + k\abs{\grad(p_1^m-p_1^{m-1})}^2
  \end{equation}
  and similarly
  \begin{equation}
    \label{3.18b}
    -2k\inner{\Delta p^m_2}{p_2^m-p_2^{m-1}} = k\abs{\grad p_2^m}^2 - k\abs{\grad p_2^{m-1}}^2 + k\abs{\grad(p_2^m-p_2^{m-1})}^2.
  \end{equation}
  \end{subequations}
  For the terms
  \begin{equation}
    \label{3.19}
    2\inner{(u^m-u^{m-1})}{p_1^m-p_1^{m-1}} +   2\inner{(v^m-v^{m-1})}{p_2^m-p_2^{m-1}},
  \end{equation}
  we consider the mapping $\mc{P}$ mentioned in Section \ref{sec:intro}:
  \[\vu=(u,v) \mapsto \vp=(p_1,p_2) = \mc{P}(\vu),\]
  as defined by \eqref{P map}. The differential of $\mc{P}$ is $\vP$. Hence the term \eqref{3.19} can be seen as
  \[2 \inner{\vu^m-\vu^{m-1}}{\mc{P}(\vu^m) -\mc{P}(\vu^{m-1})}.\]
  Now, we write:
  \begin{multline*}
    \mc{P}(\vu^m)-\mc{P}(\vu^{m-1})= \int_0^1 \f{d}{dt} \mc{P}(\vu^{m-1}+t(\vu^m-\vu^{m-1}))\,dt
    \\ =\int_0^1 \vP((1-t)\vu^{m-1}+t\vu^m)\cdot(\vu^m-\vu^{m-1})\,dt,
  \end{multline*}
  and
  \begin{equation}
    \label{3.20}
    \vp^m -\vp^{m-1} = \mc{P}(\vu^m) -\mc{P}(\vu^{m-1})  = \inner{\wbar{\vP}^m(\vu^m-\vu^{m-1})}{\vu^m-\vu^{m-1}},
  \end{equation}
  with \[\wbar{\vP}^m =\int_0^1 \vP((1-t)\vu^{m-1}+t\vu^m) \,dt.\]
  For $\vu^{m-1}\ge \vo,\, \vu^m\ge \vo$, $t\in [0,1]$, we see that $(1-t)\vu^{m-1}+t\vu^m \ge \vo$. Hence we can apply the bound \eqref{P positive definite}
  and we find
  \begin{multline*}
    \inner{\vP\big( (1-t)\vu^{m-1}+t\vu^m\big)\cdot (\vu^m-\vu^{m-1})}{\vu^m-\vu^{m-1}}
    \\\ge \alpha \big( (1-t)(u^{m-1}+v^{m-1})+t(u^m+v^m)\big)\abs{\vu^m-\vu^{m-1}}^2 + d_0 \abs{\vu^m-\vu^{m-1}}^2,
  \end{multline*}
  and
  \begin{multline*}
     \int_0^1 \inner{\vP\big( (1-t)\vu^{m-1}+t\vu^m\big)\cdot (\vu^m-\vu^{m-1})}{\vu^m-\vu^{m-1}} dt
    \\\ge \Big[ d_0+ \f{\alpha}{2} (u^{m-1}+v^{m-1}+u^m+v^m)\Big]\abs{\vu^m-\vu^{m-1}}^2 .
  \end{multline*}
  Finally \eqref{3.19} is bounded from below by
  \begin{equation}
    \label{3.21}
    \Big[ 2d_0 +\alpha(u^{m-1}+v^{m-1}+u^m+v^m)\Big]\abs{\vu^m-\vu^{m-1}}^2.
  \end{equation}
  Shifting the terms $q_1-\ell_1$ and $q_2 - \ell_2$ to the right-hand side of \eqref{SKT FD var limit}$_1$ and \eqref{SKT FD var limit}$_2$,
  we now look for an upper bound of
  \begin{equation}
    \label{3.22} -2k\inner{\vq(\vu^m)-\vl(\vu^m) }{\vp^m-\vp^{m-1}}.
  \end{equation}
  Using \eqref{3.20}, we bound the expression \eqref{3.22} from above by
  \begin{multline*}
    -2k\inner{\vq(\vu^m)-\vl(\vu^m)}{\wbar{\vP}^m(\vu^m-\vu^{m-1})}
    \\ \le \f{\alpha}{4}(u^{m-1}+v^{m-1}+u^m+v^m)\abs{\vu^m-\vu^{m-1}}^2 + ck^2\Big(\abs{\vu^{m-1}}^5+\abs{\vu^m}^5\Big).
  \end{multline*}
  We arrive at
  \begin{multline}
    \label{3.23}
    \int_\Omega \Big[ d_0 + \f{\alpha}{4}(u^{m-1}+v^{m-1}+u^m+v^m)\Big]\abs{\vu^m-\vu^{m-1}}^2\,dx
    \\+ k \abs{\grad \vp^m}^2_{L^2} - k\abs{\grad \vp^{m-1}}_{L^2}^2 +  k\abs{\grad (\vp^m-\vp^{m-1})}_{L^2}^2
   \\ \le ck^2 \Big(\abs{\vu^{m-1}}^5_{L^5}+\abs{\vu^m}^5_{L^5}\Big).
  \end{multline}
  Using exactly the same calculations as for the bounds \eqref{u5 bound}--\eqref{v5 bound}, we find
  \[\abs{\vu^{m-1}}^5_{L^5}+\abs{\vu^m}^5_{L^5} \le c(\abs{\grad \vp^{m-1}}_{L^2}^2+\abs{\grad \vp^m}_{L^2}^2+1).\]
  Thus after dividing \eqref{3.23} by $k$, we obtain
  \begin{multline}
    \label{3.25}
  k  \int_\Omega \Big[ d_0 + \f{\alpha}{4}(u^{m-1}+v^{m-1}+u^m+v^m)\Big]\abs{\f{\vu^m-\vu^{m-1}}{k}}^2\,dx
    \\+  \abs{\grad \vp^m}^2_{L^2} - \abs{\grad \vp^{m-1}}_{L^2}^2 +  \abs{\grad (\vp^m-\vp^{m-1})}_{L^2}^2
   \\ \le ck(\abs{\grad \vp^{m-1}}_{L^2}^2+\abs{\grad \vp^m}_{L^2}^2+1).
  \end{multline}
  The inequality \eqref{3.25} in particularly implies
  \begin{equation}
    \label{3.26}
    \f{\abs{\grad \vp^m}^2_{L^2} - \abs{\grad \vp^{m-1}}_{L^2}^2}{k} \le c \left( \f{\abs{\grad \vp^m}^2_{L^2} + \abs{\grad \vp^{m-1}}_{L^2}^2}{2} +1\right).
  \end{equation}
  We now apply the discrete Gronwall inequality \ref{disc gronwall} for $a_m= \abs{\grad \vp^m}^2_{L^2}, \tau_m=k, \theta=\f12, \lambda_m = g_m = c$. We note that $\omega_\ell = (1+\f k2)/(1-\f k2) = 1 + \f{2k}{2-k}$ and hence $\prod \omega_\ell \le e^T $; from this, we are able to show that $\abs{\grad \vp^m}^2_{L^2}$ is bounded uniformly for $m=1,\dots, N$ by a constant depending on $\abs{\grad \vp(\vu_0)}_{L^2}$ and $T$. In other words, we have established the  a priori estimate
  \eqref{qn strong est}.

  By using the  bound   \eqref{qn strong est} in the RHS of \eqref{3.26}, we also obtain
   \eqref{un strong est}.

   Lemma \ref{lem: dis apriori 2} is proven.
\end{proof}

\subsection{Further a priori estimates}
Similar to the estimates in  \eqref{2.10d}, \eqref{2.14b}  and \eqref{2.14c}, we also have the following a priori estimates as consequences of the bounds in \eqref{um bounds}, \eqref{strong bounds}:
\begin{Lem}\label{lem: dis apriori 3}
 With the same assumptions as in Lemma \ref{lem: dis apriori 2}, we have
\begin{subequations}
 \label{more a priori}
 \begin{equation}
   \label{qu lu est}
     k \sum_{m=1}^N  \abs{\vq(\vu^m)}_{L^2}^2 ,\quad k \sum_{m=1}^N \abs{\vl(\vu^m)}_{L^2}^2 \le c,
 \end{equation}
 \begin{equation}
   \label{lap p est}
   k \sum_{m=1}^N  \abs{\Delta\vp(\vu^m)}_{L^2}^2  \le c,
 \end{equation}
 \begin{equation}
   \label{grad u est}
 \abs{\grad \vu^m}_{L^2} \le c \text{ for } m=0,1,\dots,N.
 \end{equation}
\end{subequations}
 for $c$ depending on $T$ and $\abs{\vp(\vu_0)}_{L^2}$ but not on $N$ (nor $k$).
\end{Lem}
\begin{proof}
 Here \eqref{qu lu est} is a consequence of \eqref{um bounds1}, \eqref{um bounds4}  and \eqref{um bounds5}. In fact,
 \begin{multline*}
   k \sum_{m=1}^N  \abs{u^m}_{L^4}^4  \le (\text{H\"{o}lder inequality}) \le  k \sum_{m=1}^N \abs{u^m}_{L^2} \abs{u^m}_{L^6}^3 =   k \sum_{m=1}^N \abs{u^m}_{L^2} \abs{(u^m)^\f32}_{L^4}^2  \\\le(\text{using \eqref{um bounds1}}) \le \mc{K}_2  k \sum_{m=1}^N \abs{(u^m)^\f32}_{L^4}^2 \le(\text{Sobolev embedding for }d=4)  \\\le C  k \sum_{m=1}^N \norm{(u^m)^\f32}_{H^1}^2 \le (\text{using \eqref{um bounds4} \& \eqref{um bounds5}}) \le C,
 \end{multline*}
  and together with the similar bound for $ \dps  k \sum_{m=1}^N  \abs{v^m}_{L^4}^4  $, we find the bound for $ \dps k \sum_{m=1}^N  \abs{\vq(\vu^m)}_{L^2}^2\dps $.

  Next, the bound \eqref{lap p est} is easily obtained by using the system \eqref{SKT fd alternate}, the bound of the term $\dps d_0 (\vu^m -\vu^{m-1})/k$ in \eqref{un strong est} and the bounds of $\vq(\vu^m),\vl(\vu^m)$ in \eqref{qu lu est}.

  Finally, we infer
   \eqref{grad u est} from the estimate \eqref{qn strong est} and the relations \eqref{2.5b}, and \eqref{grad u grad p}.

\end{proof}

\subsection{Passage to the limit\label{sec:passage lim}} In this section, we pass to the limit of the system \eqref{SKT FD var limit}, or equivalently \eqref{SKT fd alternate}. We first introduce the finite difference approximate functions.
For each fixed time step $k$  we associate to the finite difference solutions $\vu^1,\vu^2,\dots,\vu^N,$  the approximate functions $\vu_k=(u_k,v_k),\,\tilde{\vu}_k=(\tilde{u}_k,\tilde{v}_k), $ as follows:
\begin{itemize}
 \item $\vu_k(t) = \vu^m,\,t\in [(m-1)k,mk]$, $m=1,\dots,N$.
 \item  $\tilde{\vu}_k(t)$ is the continuous function linear  on each time interval $[(m-1)k,mk]$ and equal to $\vu^m$ at $t=mk$, $m=0,1,\dots,N$.
\end{itemize}
The finite difference system \eqref{SKT FD var limit} is written in terms of $\vu_k,\tilde{\vu}_k$ as follows
\begin{equation}
 \label{SKT FD vector}
 \begin{cases}
   &\inner{\partial_t \tilde{\vu}_k}{\bar{\vu}} + \inner{\grad \vp(\vu_k)}{\grad \bar{\vu}} + \inner{\vq(\vu_k)}{\bar{\vu}} = \inner{\vl(\vu_k)}{\bar{\vu}},\\
   & \vu_k(0) = \vu_0 \text{ in }\Omega,
\end{cases}
\end{equation}
for all $\bar{\vu} \in V$.

Assuming only that $\vu_0\in L^2(\Omega)^2$ (and $\vu_0 \ge 0$) we infer
from \eqref{um bounds}, \eqref{strong bounds} and \eqref{more a priori} that $\vu_k,\,\tilde{\vu}_k, $ are bounded independently of $k$ as follows
\begin{subequations}\label{3.32}
\begin{align}
    &\bullet \; \label{3.32a}\abs{\tilde{\vu}_k-\vu_k}_{L^2(0,T,L^2)}^2 = \sum_{m=1}^N \abs{\f{\vu^m-\vu^{m-1}}{k}}^2 \int_{t_{m-1}}^{t_m} (t-t_m)^2dt \le \f{\mc{K}_2}{3}k,\\
  &\bullet \;\label{3.32b} \vu_k,\vu_k^\f32 \text{ belong to a bounded set in }L^2(0,T;H^1(\Omega)^2),\\ 
    &\bullet \;\label{3.32c} \tilde{\vu}_k,\tilde{\vu}_k^\f32 \text{ belong to a bounded set in }L^2(\eta,T;H^1(\Omega)^2),\forall \eta>0,\\ 
  &\bullet \; \label{3.32d} \vu_k,\tilde{\vu}_k\text{ belong to a bounded set in }L^\infty(0,T;L^2(\Omega)^2) .
\end{align}
\end{subequations}
Note that we do not use Lemma \ref{lem: dis apriori 2} at this stage because we only assume that $\vu_0\in L^2(\Omega)^2$ (and $\vu_0 \ge \vo$).

We infer from the above estimates that there exists a subsequence still denoted $k\rightarrow 0$, and $\vu,\tilde{\vu} \in L^\infty(0,T;L^2(\Omega)^2) \cap L^2(0,T;H^1(\Omega)^2)$, such that
\begin{align*}
  &\bullet \;\vu_k \rightarrow \vu \text{ in } L^\infty(0,T;L^2(\Omega)^2) \text{ weak* and }L^2(0,T;H^1(\Omega)^2)  \text{ weakly},
  \\  &\bullet \;\tilde{\vu}_k \rightarrow \tilde{\vu} \text{ in } L^\infty(0,T;L^2(\Omega)^2) \text{ weak* and }L^2(0,T;H^1(\Omega)^2) \text{ weakly},
\end{align*}
and $\vu=\tilde{\vu} $. In order to conclude that $\vu_k^\f32$ (and $\tilde{\vu}_k^\f32$) converges to $\vu^\f32$ and that $\vp(\vu_k)$ converges to $\vp(\vu)$, we proceed by compactness and derive a strong convergence result.

We rewrite \eqref{SKT fd alternate}$_1$ in the form
\begin{equation}
  \label{3.34}
  \partial_t \tilde{u}_k -\Delta p_1(\vu_k)+q_1(\vu_k) -\ell_1(u_k) = 0.
\end{equation}

We know that $\vu_k$ is bounded in $L^4(0,T;L^4(\Omega)^2)$ and $L^\infty(0,T;L^2(\Omega)^2)$ so that $q_1(\vu_k),\ell_1(u_k)$ are both bounded in $L^2(0,T;L^2(\Omega))$. The term $\Delta p_1(\vu_k)$ is written as
\[ \Delta p_1(\vu_k) = \grad \cdot \left[ (d_1+2a_{11}u_k+a_{12}v_k) \grad u_k + a_{12}u_k\grad v_k\right].\]
Considering the typical term $u_k\grad u_k$, we write
\[\int_\Omega\abs{u_k\grad u_k}^\f43 \le  \left(\int_\Omega u_k^4\right)^\f13 \left(\int_\Omega \abs{\grad u_k}^2\right)^\f23 =  \abs{\grad u_k}_{L^2}^\f43 \abs{u_k}_{L^4}^\f43,\]
and this product belongs to $L^1_t$ because the first function belongs to $L^\f32_t$ and the second one belongs to $L^3_t$.

From this we conclude that $\partial_t \tilde{u}_k$ belongs to a bounded set of
$L^\f43(0,T;W^{-1,\f43}(\Omega))$ ($W^{-1,\f43}=\grad L^\f43 =$ dual of
$W^{1,4}_0$)\footnote{see e.g. \cite[Definition 5.1]{Lio65} for the definition of the space $W^{-1,p}(\Omega)$.}. Using Aubin's compactness theorem   \ref{lem: Aubin},
we conclude that $u_k$ and $\tilde{u}_k$ converge  to $u$ in
$L^2(0,T;L^2(\Omega))$ strongly and  by an additional extraction of
subsequence, that $u_k$ converges to $u$ a.e. in $(0,T)\times
\Omega$. Then, by a standard argument see \cite[Lemma 1.3, Ch. 1]{Lio69},
we conclude that
\begin{equation}
  \label{3.35}
  u_k^\f32 \rightharpoonup u^\f32 \text{ and }p_1(\vu_k) \rightharpoonup p_1(\vu) \text{ weakly in }L^2(0,T;L^2(\Omega)).
 \end{equation}
 With the same reasoning for $v_k$, we conclude that $\vu=(u,v)$ satisfy \eqref{SKT system} or \eqref{SKT alternate}. The initial and boundary condition \eqref{bry cond and init cond} are proven in a classical way (in a weak/variational form in the case of the Neumann boundary condition).

 If in addition, we assume that $\grad \vp(\vu_0) \in L^2(\Omega)^4$, then the estimates of Lemma \ref{lem: dis apriori 2} imply by an additional extraction of subsequence that
 \begin{equation}
   \label{3.36}
   \grad \vp(\vu) \in L^\infty(0,T;L^2(\Omega)^4), \quad (1+\abs{u}+\abs{v})^\f12 \left(\abs{\partial_t u}+\abs{\partial_t v}\right) \in L^2(0,T;L^2(\Omega)),
 \end{equation}
 and returning to equations \eqref{SKT system}, we see that
 \begin{equation}
   \label{3.37}\Delta \vp(\vu) \in L^2(0,T;L^2(\Omega)^2).
 \end{equation}
 In summary we have proven the following
 \begin{Thm}[Existence of solutions]\label{thm: existence}{\color{white}s}
   \begin{enumerate}[ i) ]
     \item We assume that that $d\le 4$, that the condition \eqref{coef cond} hold, and that $\vu_0$ is given, $\vu_0\in L^2(\Omega)^2,\vu_0\ge 0$. Then equation \eqref{SKT system} possesses a solution $\vu\ge \vo$ such that, for every $T>0$:
     \begin{subequations}
       \label{3.38}
\begin{align}
  &       \vu\in L^\infty(0,T;L^2(\Omega)) \cap L^2(0,T;H^1(\Omega)^2)\\
  &(\sqrt{u} +\sqrt{v})(\abs{\grad u}+ \abs{\grad v}) \in L^2(0,T;L^2(\Omega))\\
  &\vu \in L^4(0,T;L^4(\Omega)).
\end{align}
     \end{subequations}
     with the norms in these spaces bounded by a constant depending boundedly on $T$, on the coefficients, and on the norms in $L^2(\Omega)$ of $u_0$ and $v_0$.
     \item If, in addition, $\grad \vp(\vu_0)\in L^2(\Omega)^4$, then the solution $\vu$ also satisfies \eqref{3.36} and \eqref{3.37}, with the norms in these spaces bounded by a constant deprnding boundedly on the norms of $\vu_0$ and $\grad \vp(\vu_0)$ in $L^2$ (and on $T$ and the coefficients).
   \end{enumerate}
 \end{Thm}
 \section{Attractor\label{sec:attractor}}
 Since the uniqueness of weak solutions to the SKT equation is not available, we will develop a concept of weak attractor similar to what has been done for the three-dimensional Navier-Stokes equations in \cite{Ball97}, \cite{Sell96},  \cite{FT87}, \cite{FMRT01}, \cite{FRT10}.

 We follow closely \cite{FMRT01} (see chapter III Section 4 and Appendix A5). The steps of the proof are as follows:
 \begin{itemize}
   \item[--] We define (make more precise) the concept of weak solutions
   \item[--] We derive time uniform estimates valid on $[0,\infty]$ and prove the existence of an absorbing set
   \item[--] We define the weak attractor $\mc{A}_w$, show that it is compact and that it attacks all trajectories (in a sense to be specified).
 \end{itemize}
 \subsection{Weak solutions of the SKT equations}
 Let $H=L^2(\Omega)^2$ and let $H_w$ be the space $H$ endowed with the weak topology. Here we call {\it weak} solution of the SKT equation any function $\vu$ satisfying \eqref{3.38} and \eqref{SKT system}, \eqref{bry cond and init cond} in weak (variational) form in the case of the Neumann boundary condition. We require in addition that it satisfies the following energy inequality  for all $t\ge0$:
 \begin{multline}\label{4.1}
\f12 \abs{u(t)}^2 +\f12 \abs{v(t)}^2 \\ + \int_0^t \int_\Omega   \left[ p_{11}(\vu)(\grad u)^2 + p_{12}(\vu)\grad u \grad  v + p_{21}(\vu) \grad u\grad v\ + p_{22}(\vu)(\grad v)^2 \right]dx ds
\\+ \int_0^t \int_\Omega \left[q_1(\vu)u+\ell_1(\vu)u+ q_2(\vu)u+\ell_2(\vu)v \right] dxds \le \f12 (\abs{u_0}^2 + \abs{v_0}^2).
 \end{multline}
 Note that these inequalities are satisfied by the solutions $\vu$ provided by Theorem \ref{thm: existence}. Indeed, we go back to \eqref{energy fd 1} and add this equation to the similar equation for $v$ and then pass to the lower limit $M\rightarrow \infty$. Then we reinterpret the inequality that we obtain in terms of $\vu_k$ and pass again to the lower limit $k\rightarrow 0$. For this last passage to the limit we observe that
\begin{subequations}\label{4.2}
  \begin{equation}
    \label{4.2a}
    \int_0^t\int_\Omega q_1(\vu_k)u_k \,dxds \rightarrow \int_0^t\int_\Omega q_1(\vu)u \,dxds,
  \end{equation}
  \begin{equation}
    \label{4.2b}
    \int_0^t\int_\Omega \ell_1(\vu_k)u_k \,dxds \rightarrow \int_0^t\int_\Omega \ell_1(\vu)u \,dxds,
  \end{equation}
\end{subequations}
because $\vu_k$ is bounded in $L^4(0,T;L^4(\Omega)^2)$ and $\vu_k\rightarrow \vu$ a.e. in $\Omega\times (0,T)$, using again  (\cite[Lemma 1.3, Ch. 1]{Lio69}) and the same for the terms corresponding to $q_2$ and $\ell_2$. For the other terms, we pass to the lower limit:
\begin{equation}
  \label{4.3} \liminf_{k\rightarrow 0} \abs{u_k(t)}^2 \ge \abs{u(t)}^2,\qquad \liminf_{k\rightarrow 0} \abs{v_k(t)}^2 \ge \abs{v(t)}^2,
\end{equation}
\begin{equation}
  \label{4.4}\liminf_{k\rightarrow 0} \int_0^t \int_\Omega \left(\vP(\vu_k)\grad \vu_k\right) \grad \vu_k \,dxds
  \ge
  \int_0^t \int_\Omega \left(\vP(\vu)\grad \vu \right)\grad \vu \,dxds.
\end{equation}
For \eqref{4.4}, we use the fact that $\vP$ is positive definite, see \eqref{P positive definite 2}.

\begin{Rmk}{\color{white}s} \label{rem:4.1}
  \begin{enumerate}[ i) ]
\item Observe as for $\partial_t \tilde{\vu}_k$ above, that if $\vu$ is a weak solution of the SKT equations on $(0,T)$, then $\partial_t \vu \in L^\f43(0,T;W^{-1,\f43}(\Omega)^2)$ so that $\vu$ is continuous from $[0,T]$ into $W^{-1,\f43}(\Omega)^2$ (after modification on a set of measure $0$), and since $\vu$ belongs to $L^\infty(0,T;H)$, $\vu$ is also continuous from $[0,T]$ into $H_w$:
\begin{equation}
  \label{4.5}
  \partial_t\vu \in L^\f43(0,T;W^{-1,\f43}(\Omega)^2),\quad \vu\in \mc{C}([0,T];W^{-1,\f43}(\Omega)^2)\cap \mc{C}([0,T];H_w).
\end{equation}
    \item We observe as in \cite{FMRT01} that \eqref{4.1} also holds between two times $t_1$ and $t$, instead of $0$ and $t$, with $0\le t_1<t$ for all $t's$ and for all $t_1's$ in a dense subset of $(0,t)$ of total measure. From this we deduce that
    \begin{equation}\label{4.6}
      \f{d}{dt}\mc{Y}(t) +\int_\Omega (\vP(\vu)\grad \vu) \grad \vu\,dx
      + \int_\Omega (\vq(\vu)\vu-\vl(\vu)\vu)\,dx \le 0
    \end{equation}
    where $\mc{Y}(t)=\abs{u(t)}^2+\abs{v(t)}^2$. Same proof as for (7.5) and (7.7) in \cite[Chap. II.]{FMRT01}.
       \item Concatenation: we observe that if $\vu^1$ is a weak solution of the SKT system on $(0,t_1)$ in the sense given above, and if $\vu^2$ is a weak solution on $(0,t_2)$ with $\vu^2(0)=\vu^2(t_1)$, then the function $\vu$  equal to $\vu^1$ on $(0,t_1)$ and to $\vu^2(t-t_1)$ on $(t_1,t_2-t_1)$, is a weak solution on $(0,t_1+t_2)$; see \cite{FMRT01}.
  \end{enumerate}
\end{Rmk}
\subsection{Absorbing set}
We now want to derive time uniform estimates on $(0,\infty)$ and prove the existence of an absorbing set in $H$.

With $\vP\ge 0$, and $\vu\ge \vo$, we infer from \eqref{4.6} that
\begin{equation}
  \label{4.7}
  \f{d}{dt}\mc{Y} + 2\int_\Omega(b_1u^3+c_2v^3)\,dx \le 2\int_\Omega(a_1u^2+a_2v^2)\,dx,
\end{equation}
with again $\mc{Y}(t)=\abs{u(t)}^2+\abs{v(t)}^2$. With two (four) utilizations of Young's inequality, we infer from \eqref{4.7} that
\begin{equation}
  \label{4.8}
  \mc{Y}(t)' +\alpha_1\mc{Y} \le \alpha_2,
\end{equation}
where $\alpha_1,\alpha_2$ are absolute constants.

Gronwall's lemma then implies that
\begin{equation}
  \label{4.9} \mc{Y}(t) \le \mc{Y}(0)e^{-\alpha_1 t} + \f{\alpha_2}{\alpha_1} (1 -e^{\alpha_1 t}), \quad \forall t\ge 0.
\end{equation}
This shows that $\mc{Y}(t) = \abs{u(t)}^2 +\abs{v(t)}^2$ is uniformly bounded for $t\ge 0$:
\begin{equation}
  \label{4.10}
  \mc{Y}(t) \le \mc{Y}(0)+\f{\alpha_2}{\alpha_1},
\end{equation}
and that, as $t\rightarrow \infty$
\begin{equation}
  \label{4.11}
  \limsup_{t\rightarrow \infty}\;\mc{Y}(t) \le \f{\alpha_2}{\alpha_1}.
\end{equation}
From this we deduce that the ball of $H$, $B_{2\f{\alpha_2}{\alpha_1}}(\vo)$, centered at $\vo$ of radius $2\f{\alpha_2}{\alpha_1}$ (or $r\f{\alpha_2}{\alpha_1}$, $\forall r>1$) is an {\it absorbing ball} in $H$ for the SKT system.
\subsection{The weak global attractor}
We now define the weak global attractor of the SKT equations as the set $\mc{A}_w$ of points $\vphi$ in $H$, $\vphi \ge \vo$, which belong to a complete trajectory, $\vu\ge \vo$, that is a weak solution on $\mathbb{R} $ (or on $(s,\infty)$, $\forall s\in \mathbb{R}$) of the SKT equations, with $\vu\ge \vo$. We will show that $\mc{A}_w$ is non-empty, compact in $H_w$ invariant by the flow and that it attracts all weak solutions in $H_w$. Invariant 
 by the flow means here that any trajectory, that is a weak solution of \eqref{SKT system} and starts from a point $\vu_0 \in \mc{A}_w$ is entirely included in $\mc{A}_w$.

The set $\mc{A}_w$ is not empty since it contains the point $\vo=(0,0)$. Let us show that $\mc{A}_w$ is compact in $H_w$. It follows clearly from \eqref{4.9}--\eqref{4.11} that
\begin{equation}
  \label{4.12}
  \mc{A}_w \subset B_{\f{\alpha_2}{\alpha_1}}(\vo),
\end{equation}
$\mc{A}_w$ is included in the ball of $H$ centered at $\vo$ of radius
$\alpha_2/\alpha_1$. If we show that $\mc{A}_w$ is closed in $H_w$,
then we will conclude that
\begin{equation}
  \mc{A}_w \text{ is compact in }H_w.
\end{equation}
Let $\vphi_j$ be a sequence of $\mc{A}_w$. Each $\vphi_j$ belongs to a complete trajectory $\vu^j$ (weak solution of the SKT equations on all of $\mathbb{R}$), and say $\vphi_j=\vu^j(0)$. The sequence $\vu^j$ is bounded in $H$ by $\alpha_2/\alpha_1$, according to \eqref{4.12}. We consider a sequence $t_j\rightarrow \infty$ and hence $\vu^j(-t^j)$ is bounded in $H$, and we can show that the norms of $\vu^j$ apprearing in \eqref{3.38} are bounded on $(-t_j,T)$ by a constant depending on $T$ but not on $j$. Hence we can extract from $\vu^j$ a sequence still denoted $\vu^j$ which converges to a limit $\vu$ on $(-s,T),\,\forall s,T>0$, in the sense of \eqref{3.32a}--\eqref{3.32d}, and $\vu$ is a weak solution of the SKT equations on $(-s,T),\,\forall s,T>0$ that is $\vu$ is a complete trajectory. Also $\vphi_j=\vu^j(0)$ converges weakly in $H$ to $\vphi = \vu(0)$ so that $\vphi \in \mc{A}_w$ and $\mc{A}_w$ is closed in $H_w$.

The invariance of $\mc{A}_w$ follows from the concatenation property mentioned in Remark \ref{rem:4.1}, iii). Consider a weak solution $\vu$ starting from $\vu_0=\vu(0)\in \mc{A}_w$. Since $\vu_0\in \mc{A}_w$ it belongs to a complete trajectory $\tilde{\vu}$ with say $\tilde{\vu}(0)=\vu_0$. By the concatenation property mentioned in Remark \ref{rem:4.1}, iii) the function $\vu^*$ equal to $\tilde{\vu}$ for $t\le 0$ and to $\tilde{\vu}$ for $t\ge 0$ is a complete trajectory and it is therefore included in $\mc{A}_w$.

There remains to show that $\mc{A}_w$ attracts all weak solutions in $H_w$ as $t\rightarrow \infty$. We proceed as in \cite{FMRT01}. We will show a stronger result, namely that $\mc{A}_w$ attracts all the solutions in the weak topology of $H$, {\it uniformly} for all the initial conditions in a bounded set of $H$. Indeed, consider a sequence of initial data $\vu_{0n}$ bounded in $H$ (with $\vu_{0n}\ge \vo$):
\begin{equation}
  \label{4.14}
  \abs{\vu_{0n}} \le c_1 \text{ for all }n\in \mathbb{N},
\end{equation}
and consider the corresponding weak solutions provided by Theorem \ref{thm: existence} i), $\left(\vu_n\right)_{n\in\mathbb{N}}$. Consider a neighborhood $\mc{U}$ of $\mc{A}_w$ in $H_w$. We will show that there exists $t_1=t_1(\mc{U})$, such that $\vu_n(t) \in \mc{U}$, $\forall t\ge t_1$ and $\forall n\in \mathbb{N}$.

The proof of this result is by contradiction. Assume the property is not true: then there exists a neighborhood $\mc{U}$ of $\mc{A}_w$ in $H_w$ and two sequences $\{n_j\}_j,n_j\in\mathbb{N}$ and $\{t_j\}_j$, $t_j\rightarrow \infty$ such that $\vu_{n_j}(t_j) $ does not belong to $\mc{U}$. We deduce from \eqref{4.14} and \eqref{4.10}  that the sequence $\vu_{n_j}(t_j)$ is bounded in $H$. Hence extracting a subsequence from $\{n_j\}_j$ and $\{t_j\}_j$ (still denoted by $j$), there exists $\vv_0 \in H$ with of course $\vv_0 \ge \vo$, such that
\begin{equation}
  \label{4.15}
  \vu_{n_j}(t_j) \rightharpoonup \vv_0 \text{ weakly in }H.
\end{equation}
Since $\vu_{n_j}(t_j) \not\in \mc{U}$ and since $\mc{U}$ is a neighborhood of $\mc{A}_w$, it follows that $\vv_0\not \in \mc{A}_w$. We now show that $\vv_0$ belongs to a complete trajectory, so that $\vv_0 \in \mc{A}_w$ thus establishing a contradiction. Indeed, define $\vv_j(t) = \vu_{n_j}(t+t_j),$ for $t\ge -t_j$. It is clear that $\vv_j$ is a weak solution of \eqref{SKT system} on $(-t_j,\infty)$, with $\vv_j(0)=\vu_{n_j}(t_j)$. The sequence $\vv_j$ is bounded in $L^\infty(-t_j,\infty;H)$ and satisfies a priori estimates similar to (\ref{3.38}a--c), and these a priori estimates are independant of $j$ because $\vv_j(0)=\vu_{n_j}(t_j)$ is bounded in $H$. We deduce that there exists a subsequence, still denoted by $j$, which converges to $\vv$ on any interval $(-s,T)$ in the sense of (\ref{3.38}a--c), and
$\vv\ge \vo$ is a complete trajectory and $\vv(0)=\lim \vv_j(0)=\vv_0$ by \eqref{4.15} (and $\vv_j(0)=\vu_{n_j}(t_j)$). The contradiction is established and we have shown that $\mc{A}_w$ attracts all trajectories.

In summary we have proven the following
\begin{Thm}\label{thm: attractor}
  Assume that \eqref{coef cond} holds. Then the set $\mc{A}_w$ of all solutions $\vu\ge \vo$ of \eqref{SKT system} on all of $\mathbb{R}$  belonging to $L^\infty(\mathbb{R};H)$ is non-empty, compact in $H_w$, invariant by the flow and it attracts all trajectories in $H_w$, uniformly for all the initial data $\vu_0$ in a bounded set of $H$.
\end{Thm}
\nocite{Ton07,TW06}
\paragraph{Acknowledgement.} This work was supported in part by NSF grant DMS151024 and by the Research Fund of Indiana University. The authors thank Ricardo Rosa for very useful discussions.
\appendix
\section{Appendix}

The following discrete Gronwall lemma can be found in e.g. \cite[Lemma 3.2]{Emm99}, \cite{TW06}, \cite{Ton07}:
\begin{Lem}\label{lem:disc Gronwall}
  Let $\{a_n\}, \{g_n\} \subset \mathbb{R}, \{g_n\} \subset \mathbb{R}^+$ be such that
  \[\f{a_n-a_{n-1}}{\tau_n} \le g_n + (1-\theta)\lambda_{n-1} a_{n-1} + \theta\lambda_na_n,\qquad n=1,2 ,\dots\]
  If $(1-\theta \tau_n \lambda_n)>0, \, 1+(1-\theta)\lambda_{n-1}\tau_n>0, \;n=1,2 ,\dots$ then
  \begin{equation}
    \label{disc gronwall}
    a_n \le a_0 \prod_{\ell=1}^n \omega_\ell + \sum_{j=0}^{n-1} \f{\tau_{j+1}g_{j+1}}{1+(1-\theta)\lambda_j\tau_{j+1}} \prod_{\ell=j+1}^n \omega_\ell ,
  \end{equation}
  where $\dps \omega_\ell = \f {1+(1-\theta)\lambda_\ell \tau_{\ell+1}}{1-\theta \lambda_{\ell+1}\tau_{\ell+1}} $.
\end{Lem}

The following Aubin-Lions compactness result appears in e.g. \cite{Lio69} or \cite{Tem01} :

\begin{Lem} \label{lem: Aubin}Let $X_0,X$ and $X_1$ be three Banach spaces such that
$X_0 \subseteq X \subseteq X_1$ and $X_i$, $i=1,2$ are reflexive. Suppose that $X_0$ is compactly embedded in $X$ and that $X$ is continuously embedded in $X_1$.  For $p,q>1$, let
\[\mathcal{X} = \{ u\in L^p(0,T;X_0) \text{ such that } \dot{u} \in L^q(0,T;X_1)\}.\]
Then the embedding of $\mathcal{X}$ into $L^p(0,T;X)$ is compact.
\end{Lem}

\bibliographystyle{apalike}
\bibliography{MP}

\def\cfudot#1{\ifmmode\setbox7\hbox{$\accent"5E#1$}\else
  \setbox7\hbox{\accent"5E#1}\penalty 10000\relax\fi\raise 1\ht7
  \hbox{\raise.1ex\hbox to 1\wd7{\hss.\hss}}\penalty 10000 \hskip-1\wd7\penalty
  10000\box7} \def\cprime{$'$}
\begin{thebibliography}{}

\bibitem[Amann, 1989]{Ama89}
Amann, H. (1989).
\newblock Dynamic theory of quasilinear parabolic systems. {III}. {G}lobal
  existence.
\newblock {\em Math. Z.}, 202(2):219--250.

\bibitem[Amann, 1990]{Ama90}
Amann, H. (1990).
\newblock Dynamic theory of quasilinear parabolic equations. {II}.
  {R}eaction-diffusion systems.
\newblock {\em Differential Integral Equations}, 3(1):13--75.

\bibitem[Ball, 1997]{Ball97}
Ball, J.~M. (1997).
\newblock Continuity properties and global attractors of generalized semiflows
  and the {N}avier-{S}tokes equations.
\newblock {\em J. Nonlinear Sci.}, 7(5):475--502.

\bibitem[Choi et~al., 2004]{CLY04}
Choi, Y.~S., Lui, R., and Yamada, Y. (2004).
\newblock Existence of global solutions for the
  {S}higesada-{K}awasaki-{T}eramoto model with strongly coupled
  cross-diffusion.
\newblock {\em Discrete Contin. Dyn. Syst.}, 10(3):719--730.

\bibitem[Emmrich, 1999]{Emm99}
Emmrich, E. (1999).
\newblock {\em Discrete versions of Gronwall's lemma and their application to
  the numerical analysis of parabolic problems}.
\newblock Preprint Reihe Mathematik. TU, Fachbereich 3.

\bibitem[Foias et~al., 2001]{FMRT01}
Foias, C., Manley, O., Rosa, R., and Temam, R. (2001).
\newblock {\em Navier-{S}tokes equations and turbulence}, volume~83 of {\em
  Encyclopedia of Mathematics and its Applications}.
\newblock Cambridge University Press, Cambridge.

\bibitem[Foias et~al., 2010]{FRT10}
Foias, C., Rosa, R., and Temam, R. (2010).
\newblock Topological properties of the weak global attractor of the
  three-dimensional {N}avier-{S}tokes equations.
\newblock {\em Discrete Contin. Dyn. Syst.}, 27(4):1611--1631.

\bibitem[Foias and Temam, 1987]{FT87}
Foias, C. and Temam, R. (1987).
\newblock The connection between the {N}avier-{S}tokes equations, dynamical
  systems, and turbulence theory.
\newblock In {\em Directions in partial differential equations ({M}adison,
  {WI}, 1985)}, volume~54 of {\em Publ. Math. Res. Center Univ. Wisconsin},
  pages 55--73. Academic Press, Boston, MA.

\bibitem[Hoang et~al., 2015]{HNP15}
Hoang, L.~T., Nguyen, T.~V., and Phan, T.~V. (2015).
\newblock Gradient estimates and global existence of smooth solutions to a
  cross-diffusion system.
\newblock {\em SIAM J. Math. Anal.}, 47(3):2122--2177.

\bibitem[Le, 2013]{Le13}
Le, D. (2013).
\newblock Regularity of {BMO} weak solutions to nonlinear parabolic systems via
  homotopy.
\newblock {\em Trans. Amer. Math. Soc.}, 365(5):2723--2753.

\bibitem[Le, 2016]{Le15}
Le, D. (2016).
\newblock Weighted {G}agliardo-{N}irenberg inequalities involving {BMO} norms
  and solvability of strongly coupled parabolic systems.
\newblock {\em Adv. Nonlinear Stud.}, 16(1):125--146.

\bibitem[Le et~al., 2003]{LNN03}
Le, D., Nguyen, L.~V., and Nguyen, T.~T. (2003).
\newblock Shigesada-{K}awasaki-{T}eramoto model on higher dimensional domains.
\newblock {\em Electron. J. Differential Equations}, pages No. 72, 12.

\bibitem[Lions, 1965]{Lio65}
Lions, J.~L. (1965).
\newblock {\em Probl\`emes aux limites dans les \'equations aux d\'eriv\'ees
  partielles}, volume 1962 of {\em Deuxi\`eme \'edition. S\'eminaire de
  Math\'ematiques Sup\'erieures, No. 1 (\'Et\'e}.
\newblock Les Presses de l'Universit\'e de Montr\'eal, Montreal, Que.

\bibitem[Lions, 1969]{Lio69}
Lions, J.-L. (1969).
\newblock {\em Quelques m\'ethodes de r\'esolution des probl\`emes aux limites
  non lin\'eaires}.
\newblock Dunod; Gauthier-Villars, Paris.

\bibitem[Lou et~al., 1998]{LNW98}
Lou, Y., Ni, W.-M., and Wu, Y. (1998).
\newblock On the global existence of a cross-diffusion system.
\newblock {\em Discrete Contin. Dynam. Systems}, 4(2):193--203.

\bibitem[Nguyen, 2006]{Ngu06}
Nguyen, T. (2006).
\newblock {\em Cross Diffusion Systems}.
\newblock University of Texas at San Antonio.
\newblock Thesis (Master)--UTSA.

\bibitem[Sell, 1996]{Sell96}
Sell, G.~R. (1996).
\newblock Global attractors for the three-dimensional {N}avier-{S}tokes
  equations.
\newblock {\em J. Dynam. Differential Equations}, 8(1):1--33.

\bibitem[Shigesada et~al., 1979]{SKT79}
Shigesada, N., Kawasaki, K., and Teramoto, E. (1979).
\newblock Spatial segregation of interacting species.
\newblock {\em J. Theoret. Biol.}, 79(1):83--99.

\bibitem[Temam, 1997]{Tem97}
Temam, R. (1997).
\newblock {\em Infinite-dimensional dynamical systems in mechanics and
  physics}, volume~68 of {\em Applied Mathematical Sciences}.
\newblock Springer-Verlag, New York, second edition.

\bibitem[Temam, 2001]{Tem01}
Temam, R. (2001).
\newblock {\em Navier-{S}tokes equations}.
\newblock AMS Chelsea Publishing, Providence, RI.
\newblock Theory and numerical analysis, Reprint of the 1984 edition.

\bibitem[Tone, 2007]{Ton07}
Tone, F. (2007).
\newblock On the long-time stability of the implicit {E}uler scheme for the
  2{D} space-periodic {N}avier-{S}tokes equations.
\newblock {\em Asymptot. Anal.}, 51(3-4):231--245.

\bibitem[Tone and Wirosoetisno, 2006]{TW06}
Tone, F. and Wirosoetisno, D. (2006).
\newblock On the long-time stability of the implicit {E}uler scheme for the
  two-dimensional {N}avier-{S}tokes equations.
\newblock {\em SIAM J. Numer. Anal.}, 44(1):29--40.

\bibitem[Tu{\cfudot{o}}c, 2007]{Tuo07}
Tu{\cfudot{o}}c, P.~V. (2007).
\newblock Global existence of solutions to {S}higesada-{K}awasaki-{T}eramoto
  cross-diffusion systems on domains of arbitrary dimensions.
\newblock {\em Proc. Amer. Math. Soc.}, 135(12):3933--3941 (electronic).

\bibitem[Yagi, 1988]{Yag88}
Yagi, A. (1988).
\newblock Fractional powers of operators and evolution equations of parabolic
  type.
\newblock {\em Proc. Japan Acad. Ser. A Math. Sci.}, 64(7):227--230.

\bibitem[Yagi, 1990]{Yag90}
Yagi, A. (1990).
\newblock Parabolic evolution equations in which the coefficients are the
  generators of infinitely differentiable semigroups. {II}.
\newblock {\em Funkcial. Ekvac.}, 33(1):139--150.

\bibitem[Yagi, 1991]{Yag91}
Yagi, A. (1991).
\newblock Abstract quasilinear evolution equations of parabolic type in
  {B}anach spaces.
\newblock {\em Boll. Un. Mat. Ital. B (7)}, 5(2):341--368.

\bibitem[Yagi, 1993]{Yag93}
Yagi, A. (1993).
\newblock Global solution to some quasilinear parabolic system in population
  dynamics.
\newblock {\em Nonlinear Anal.}, 21(8):603--630.

\end{thebibliography}

\end{document}